\documentclass{CSML}
\pdfoutput=1
\usepackage{amssymb,amsthm}
\usepackage{amsfonts}
\usepackage{amsmath}
\usepackage[british]{babel}
\usepackage[hidelinks]{hyperref}
\usepackage{cleveref}
\usepackage{mathtools}

\usepackage{lastpage}
\newcommand{\dOi}{13(2:10)2017}
\lmcsheading{}{1--\pageref{LastPage}}{}{}{Dec.~07, 2016}{Jun.~22, 2017}{}

\def\smallbullet{\mbox{\larger[-5]$\bullet$}}
\def\lowsmile{\raisebox{-0.45ex}{$\smile$}}
\def\cupdot {\stackrel{\smallbullet}{\cup}}

\def\minusdot {\stackrel{\smallbullet}{\setminus}}
\def\negdot{\stackrel{\smallbullet}{-}}
\def\dcupdot {\stackrel{\smallbullet}{\lowsmile}}
\def\plusdot {\stackrel{\smallbullet}{\sqcup}}

\def\restr #1{{\restriction_{#1}}}

\def\set#1{\{#1\}}
\def\c#1{\mathfrak #1}
\def\lang#1{\mathcal #1}
\def\class#1{\mathbf#1}
\def\defn#1{\textbf{#1}}
\def\powerset{{\raisebox{0.4ex}{$\wp$}}}

\def\biimp{\leftrightarrow}
\def \impl {\rightarrow}
\def \bmeet {\cdot}
\def \wand {\mathbin{-\!*}}

\newcommand{\compo}{\mathbin{;}}

\newtheorem{theorem}{Theorem}[section]
\newtheorem{proposition}[theorem]{Proposition}
\newtheorem{lemma}[theorem]{Lemma}
\newtheorem{corollary}[theorem]{Corollary}

\theoremstyle{definition}
\newtheorem{remark}[theorem]{Remark}
\newtheorem{problem}[theorem]{Problem}
\newtheorem{definition}[theorem]{Definition}

\newtheorem{example}[theorem]{Example}
\def\ws{{winning strategy}}

\def\Ax#1{{\bf Ax$(#1)$}}

\def\lr#1{{\langle#1\rangle}}
 
\def\Los{\L{o}\'{s}}

\title{Disjoint-union Partial Algebras}
\author{Robin Hirsch}
\author{Brett McLean}
\address{Department of Computer Science, University College London, Gower Street, London WC1E 6BT}
\email{\{r.hirsch,b.mclean\}@ucl.ac.uk}

\begin{document}

\begin{abstract}
Disjoint union is a partial binary operation returning the union of two sets if they are disjoint and undefined otherwise. A disjoint-union partial algebra of sets is a collection of sets closed under disjoint unions, whenever they are defined. We provide a recursive first-order axiomatisation of the class of partial algebras isomorphic to a disjoint-union partial algebra of sets but prove that no finite axiomatisation exists. We do the same for other signatures including one or both of disjoint union and subset complement, another partial binary operation we define.

Domain-disjoint union is a partial binary operation on partial functions, returning the union if the arguments have disjoint domains and undefined otherwise. For each signature including one or both of domain-disjoint union and subset complement and optionally including composition, we consider the class of partial algebras isomorphic to a collection of partial functions closed under the operations. Again the classes prove to be axiomatisable, but not finitely axiomatisable, in first-order logic.

We define the notion of pairwise combinability. For each of the previously considered signatures, we examine the class  isomorphic to a partial algebra of sets/partial functions under an isomorphism mapping arbitrary suprema of pairwise combinable sets to the corresponding disjoint unions. We prove that for each case the class is not closed under elementary equivalence.  

However,  when intersection is added to any of the signatures considered, the isomorphism class of the partial algebras of sets is finitely axiomatisable and in each case we give such an axiomatisation.
\end{abstract}

\maketitle

\section{Introduction}

Sets and functions are perhaps the two most fundamental and important types of object in all mathematics. Consequently, investigations into the first-order properties of collections of such objects have a long history. Boole, in 1847, was the first to focus attention directly on the algebraic properties of sets \cite{B48}. The outstanding result in this area is the Birkhoff-Stone representation theorem, completed in 1934, showing that boolean algebra provides a first-order axiomatisation of the class of isomorphs of fields of sets \cite{Sto34}.

For functions, the story starts around the same period, as we can view Cayley's theorem of 1854 as proof that the group axioms are in fact an axiomatisation of the isomorphism class of collections of bijective functions, closed under composition and inverse \cite{doi:10.1080/14786445408647421}. Schein's survey article of 1970 contains a summary of the many similar results about algebras of partial functions that were known by the time of its writing \cite{Schein1970}.

The past fifteen years have seen a revival of interest in algebras of partial functions,  with results finding that such algebras are logically and computationally well behaved \cite{invitation, Jackson2003393, 1182.20058, DBLP:journals/ijac/JacksonS11, hirsch, completerep}. In particular, algebras of partial functions with composition, intersection, domain and range have the finite representation property \cite{finiterep}.

Separation logic is a formalism for reasoning about the state of dynamically-allocated computer memory \cite{reynolds2002separation}. In the standard `stack-and-heap' semantics, dynamic memory states are modelled by (finite) partial functions. Thus statements in separation logic are statements about partial functions.

The logical connective common to all flavours of separation logic is the separating conjunction $*$. In the stack-and-heap semantics, the formulas are evaluated at a given heap (a partial function, $h$) and stack (a variable assignment, $s$). In this semantics $h, s \models \varphi * \psi$ if and only if there exist  $h_1, h_2$  with disjoint domains, such that $h = h_1 \cup h_2$ and $h_1,s \models \varphi$ and $h_2, s \models \psi$. So lying behind the semantics of the separating conjunction is a partial operation on partial functions we call the domain-disjoint union, which returns the union when its arguments have disjoint domains and is undefined otherwise. Another logical connective that is often employed in separation logic is the separating implication and again a partial operation on partial functions lies behind its semantics.

Separation logic has enjoyed and continues to enjoy great practical successes \cite{berdine2005smallfoot, calcagno2015moving}. However Brotherston and Kanovich have shown that, for propositional separation logic, the validity problem is undecidable for a variety of different semantics, including the stack-and-heap semantics \cite{Brotherston-Kanovich:10}.  The contrast between the aforementioned positive results concerning  algebras of partial functions and the undecidability of a propositional logic  whose semantics are based on partial algebras of partial functions, suggests a more detailed investigation into the computational and logical behaviour of collections of partial functions equipped with the partial operations arising from separation logic. 

In this paper we examine, from a first-order perspective, partial algebras of partial functions over separation logic signatures---signatures containing one or more of the partial operations underlying the semantics of separation logic. Specifically, we study, for each signature, the isomorphic closure of the class of partial algebras of partial functions. Because these partial operations have not previously been studied in a first-order context we also include an investigation into partial algebras of \emph{sets} over these signatures.

In \Cref{definitions} we give the definitions needed to precisely define these classes of partial algebras. In \Cref{Axiomatisability} we show that each of our classes is first-order axiomatisable and in \Cref{Recursive} we give a method to form recursive axiomatisations that are easily understandable as statements about certain two-player games.

In \Cref{Non-axiomatisability} we show that though our classes are axiomatisable, finite axiomatisations do not exist.   	In \Cref{meet} we show that when ordinary intersection is added to the previously examined signatures, the classes of partial algebras become finitely axiomatisable.
In \Cref{Decidability} we examine decidability and complexity questions and then conclude with some open problems.

\section{Disjoint-union Partial Algebras}\label{definitions}

In this section we give the fundamental definitions that are needed in order to state the results contained in this paper. We first define the partial operations that we use.

\begin{definition}
Given two sets $S$ and $T$ the \defn{disjoint union} $S\cupdot T$ equals $S\cup T$ if $S\cap T=\emptyset$, else it is undefined. The \defn{subset complement} $S \minusdot T$ equals $S \setminus T$ if $T \subseteq S$, else it is undefined.
\end{definition}
Observe that $S\cupdot T=U$ if and only if $U\minusdot S=T$.

The next definition involves partial functions. We take the set-theoretic view of a function as being a functional set of ordered pairs, rather than requiring a domain and codomain to be explicitly specified also. In this sense there is no notion of a function being `partial'. But using the word partial serves to indicate that when we have a set of such functions they are not required to share a common domain (of definition)---they are `partial functions' on (any superset of) the union of these domains.  

\begin{definition}\label{dd}
Given two partial functions $f$ and $g$ the \defn{domain-disjoint union} $f\dcupdot g$ equals $f\cup g$ if the domains of $f$ and $g$ are disjoint, else it is undefined. The symbol $\mid$ denotes the total operation of composition.
\end{definition}

Observe that if the domains of two partial functions are disjoint then their union is a partial function. So domain-disjoint union is a partial operation on partial functions. If $f$ and $g$ are partial functions with $g \subseteq f$ then $f \setminus g$ is also a partial function. Hence subset complement gives another partial operation on partial functions.

The reason for our interest in these partial operations is their appearance in the semantics of separation logic, which we now detail precisely.

The \defn{separating conjunction} $*$ is a binary logical connective present in all forms of separation logic. As mentioned in the introduction, in the stack-and-heap semantics the formulas are evaluated at a given heap (a partial function, $h$) and stack (variable assignment, $s$). In this semantics $h, s \models \varphi * \psi$ if and only if there exist  $h_1, h_2$ such that $h = h_1 \dcupdot h_2$ and both $h_1,s \models \varphi$ and $h_2, s \models \psi$.

The constant $\texttt{emp}$ also appears in all varieties of separation logic. The semantics is $h, s \models \texttt{emp}$ if and only if $h = \emptyset$.

The \defn{separating implication} $\wand$ is another binary logical connective common in separation logic.  The semantics is $h, s \models \varphi \wand \psi$ if and only if for all $h_1, h_2$ such that $h = h_2 \minusdot h_1$ we have $h_1, s \models \varphi $ implies  $h_2, s \models \psi$.

\medskip

Because we are working with partial operations, the classes of structures we will examine are classes of partial algebras.

\begin{definition}\label{palgebra}
A \defn{partial algebra} $\c A = (A, (\Omega_i)_{i < \beta})$ consists of a domain, $A$, together with a sequence $\Omega_0, \Omega_1, \ldots$ of partial operations on $A$, each of some finite arity $\alpha(i)$ that should be clear from the context. Two partial algebras $\c A=(A, (\Omega_i)_{i < \beta})$ and $ \c B=(B, (\Pi_i)_{i < \beta})$ are \defn{similar} if for all $i < \beta$ the arities of $\Omega_i$ and $\Pi_i$ are equal. (So in particular $\c A$ and $\c B$ must have the same ordinal indexing their partial operations.)
\end{definition}

We use the word `signature' flexibly. Depending on context it either means a sequence of symbols, each with a prescribed arity and each designated to be a function symbol, a partial function symbol or a relation symbol. Or, it means a sequence of actual operations/partial operations/relations.

\begin{definition}\label{homo}
Given two similar partial algebras $\c A=(A, (\Omega_i)_{i < \beta})$ and $ \c B=(B, (\Pi_i)_{i < \beta})$, a map $\theta:A\rightarrow B$ is a \defn{partial-algebra homomorphism} from $\c A$ to $\c B$  if for all $i < \beta$ and all $a_1, \ldots, a_{\alpha(i)} \in A$ the value $\Omega_i(a_1, \ldots, a_{\alpha(i)})$ is defined if and only if $\Pi_i(\theta(a_1), \ldots, \theta(a_{\alpha(i)}))$ is defined, and in the case where they are defined we have \[\theta(\Omega_i(a_1, \ldots, a_{\alpha(i)}))=\Pi_i(\theta(a_1), \ldots, \theta(a_{\alpha(i)}))\]. If $\theta$ is surjective then we say $\c B$ is a \defn{partial-algebra homomorphic image} of $\c A$. A \defn{partial-algebra embedding} is an injective partial-algebra homomorphism.  An \defn{isomorphism} is a bijective partial-algebra homomorphism. 
\end{definition}

We are careful never to drop the words `partial-algebra' when referring to the notions defined in \Cref{homo}, since a bald `homomorphism' is an ambiguous usage when speaking of partial algebras---at least three differing definitions have been given in the literature. What we call a partial-algebra homomorphism, Gr\"{a}tzer calls a strong homomorphism \cite[Chapter~2]{Gratzer:ua79}.

Given a partial algebra $\c A$, when we write $a \in \c A$ or say that $a$ is an element of $\c A$, we mean that $a$ is an element of the domain of $\c A$. While total algebras are by convention nonempty, we make the choice, for reasons of convenience, to allow partial algebras to be empty. When we want to refer to a signature consisting of a single symbol we will often abuse notation by using that symbol to denote the signature.

We write $\powerset(X)$ for the power set of a set $X$.

\begin{definition}\label{power}
Let $\sigma$ be a signature whose symbols are members of $\{\cupdot, \minusdot, \emptyset\}$. A \defn{partial $\sigma$-algebra of sets}, $\c A$, with domain $A$, consists of a subset $A\subseteq \powerset(X)$ (for some \defn{base} set $X$), closed under the partial operations in $\sigma$, wherever they are defined, and containing the empty set if $\emptyset$ is in the signature. The particular case of $\sigma = (\cupdot)$ is called a \defn{disjoint-union partial algebra of sets} and the case $\sigma = (\cupdot, \emptyset)$ is a disjoint-union partial algebra of sets \defn{with zero}.
\end{definition}

\begin{definition}
Let $\sigma$ be a signature whose symbols are members of $\{\dcupdot, \minusdot, \mid, \emptyset\}$. A \defn{partial $\sigma$-algebra of partial functions}, $\c A$ consists of a set of partial functions closed under the partial and total operations in $\sigma$, wherever they are defined, and containing the empty set if $\emptyset$ is in the signature. The \defn{base} of $\c A$ is the union of the domains and codomains of all the partial functions in $\c A$.
\end{definition}

\begin{definition}
Let $\sigma$ be a signature whose symbols are members of $\{\cupdot, \minusdot, \emptyset\}$. A \defn{$\sigma$-representation by sets} of a partial algebra is an isomorphism from that partial algebra to a partial $\sigma$-algebra of sets. The particular case of $\sigma = (\cupdot)$ is called a \defn{disjoint-union representation} (by sets). 
\end{definition}

\begin{definition}
Let $\sigma$ be a signature whose symbols are members of $\{\dcupdot, \minusdot, \mid, \emptyset\}$. A \defn{$\sigma$-representation by partial functions} of a partial algebra is an isomorphism  from that partial algebra to a partial $\sigma$-algebra of partial functions.  
\end{definition}

For a partial algebra $\c A$ and an element $a\in\c A$, we write $a^\theta$ for the image of $a$ under a representation $\theta$ of $\c A$. We will be consistent about the symbols we use for abstract (partial) operations---those in the partial algebras being represented---employing them according to the correspondence indicated below.
\begin{align*}
{\plusdot} \quad&\leadsto \quad\mbox{$\cupdot$ \;or\; $\dcupdot$}\\
{\negdot}\quad &\leadsto\quad {\minusdot}\\
{;\,} \quad&\leadsto \quad{\mid}\\
{0}\quad &\leadsto\quad {\emptyset}
\end{align*}

For each notion of representability we are interested in the associated representation class---the class of all partial algebras having such a representation. It is usually clear whether we are talking about a representation by sets or a representation by partial functions. For example if the signature contains $\cupdot$ we must be talking of sets and if it contains $\dcupdot$  we must be talking of partial functions. However, as part \eqref{first} of the next proposition shows, for the partial operations we are considering, representability by sets and representability by partial functions are the same thing.

\begin{proposition}\label{prop} \leavevmode
\begin{enumerate}\item  \label{first} 
Let $\sigma$ be a signature whose symbols are a subset of $\set{\cupdot, \minusdot, \emptyset}$ and let $\sigma'$ be the signature formed by replacing $\cupdot$ (if present) by $\dcupdot$ in $\sigma$. A partial algebra is $\sigma$-representable by sets if and only if it is $\sigma'$-representable by partial functions.
\item\label{mid}
 Let  $\c A$ be a partial $(\plusdot, \compo, 0)$-algebra. If the $(\plusdot, 0)$-reduct of $\c A$ is $(\cupdot, \emptyset)$-representable and $\c A$ validates $a \compo b = 0$, then $\c A$ is $(\dcupdot, \mid, \emptyset)$-representable.
\end{enumerate}
\end{proposition}

\begin{proof}
For part \eqref{first}, let $\sigma$ be one of the signatures in question and let $\c A$ be a partial algebra. Suppose  $\theta$ is a $\sigma$-representation of $\c A$ by sets over base  $X$.  Then the map $\rho$ defined by $a^\rho=\set{(x, x)\mid x\in a^\theta}$ is easily seen to be a $\sigma'$-representation of $\c A$ by partial functions. 

 Conversely, suppose $\rho$ is a $\sigma'$-representation of $\c A$ by partial functions over base $X$. Let $Y$ be a disjoint set of the same cardinality as $X$ and let $f:X\rightarrow Y$ be any bijection.  Define $\theta$ by $a^\theta=a^\rho\cup \set{(f(x),f(x))\mid x\in \operatorname{dom}(a^\rho)}$. Then it is easy to see that $\theta$ is \emph{another} $\sigma'$-representation of $\c A$ by partial functions. By construction, $\theta$ has the property that any $a^\theta$ and $b^\theta$  have disjoint domains if and only if they are disjoint. Hence $\theta$ is also a $\sigma$-representation of $\c A$ by sets.

 For part \eqref{mid}, let $\theta$ be a $(\cupdot, \emptyset)$-representation of the $(\plusdot, 0)$-reduct of $\c A$ over base set $X$. Let $Y$ be a disjoint set of the same cardinality as $X$ and let $f:X\rightarrow Y$ be any bijection.  The map $\rho$ defined by $a^\rho=f\restr{a^\theta}$ is easily seen to be a $(\dcupdot, \mid, \emptyset)$-representation of $\c A$.
\end{proof}

\begin{remark}\label{add0}In each of the following cases let the signature $\sigma_\emptyset$ be formed by the addition of $\emptyset$ to $\sigma$.
\begin{itemize}
\item
Let $\sigma$ be a signature containing $\cupdot$. A partial algebra $\c A$ is $\sigma$-representable if and only its reduct to the signature without $0$ is $\sigma_\emptyset$-representable and $\c A$ satisfies $0 \plusdot 0 = 0$.
\item
Let $\sigma$ be a signature containing $\dcupdot$. A partial algebra $\c A$ is $\sigma$-representable if and only its reduct to the signature without $0$ is $\sigma_\emptyset$-representable and $\c A$ satisfies $0 \plusdot 0 = 0$.
\item
Let $\sigma$ be a signature containing $\minusdot$. A partial algebra $\c A$ is $\sigma$-representable if and only its reduct to the signature without $0$ is $\sigma_\emptyset$-representable and $\c A$ satisfies $0 \negdot 0 = 0$.
\end{itemize}
Hence axiomatisations of representation classes for signatures without $\emptyset$ would  immediately yield axiomatisations for the case including $\emptyset$ also.
\end{remark}

We now define a version of complete representability. For a partial $(\plusdot, \ldots)$-algebra $\c A$, define a relation $\lesssim$ over $\c A$ by letting $a\lesssim b$ if and only if  either $a=b$ or there is $c\in\c A$ such that $a\plusdot c$ is defined and $a\plusdot c=b$.  By definition, $\lesssim$ is reflexive. If $\c A$ is $(\cupdot, \ldots)$-representable, then by elementary properties of sets, it is necessarily the case that if $(a\plusdot b)\plusdot c$ is defined then $a\plusdot(b\plusdot c)$ is also defined and equal to it, which is precisely what is required to see that $\lesssim$ is transitive. Antisymmetry of $\lesssim$ also follows by elementary properties of sets. Hence $\lesssim$ is a partial order.

\begin{definition}\label{completely}
A subset $S$ of a partial $(\plusdot, \ldots)$-algebra $\c A$ is \defn{pairwise combinable} if for all $s\neq t\in S$ the value $s\plusdot t$ is defined.  
A $(\cupdot, \ldots)$-representation $\theta$ of $\c A$ is \defn{$\lesssim$-complete} if for any pairwise-combinable subset $S$ of $\c A$ with a supremum $a$ (with respect to the order $\lesssim$) we have $a^\theta = \bigcup_{s\in S}s^\theta$.
\end{definition}

\begin{proposition}\label{prop:complete}
If $\c A$ is a finite partial $(\plusdot, \ldots)$-algebra then every $(\cupdot, \ldots)$-representation of  $\c A$ is $\lesssim$-complete.
\end{proposition}

\begin{proof}
Let $\theta$ be a $(\cupdot, \ldots)$-representation of $\c A$ and let $S \subseteq \c A$ be pairwise combinable with supremum $a$. As $S$ is pairwise combinable and $\theta$ is a $(\cupdot, \ldots)$-representation, we have that $s^\theta \cupdot t^\theta$ is defined for all $s \neq t \in S$. Then by the definition of $\cupdot$, the set $\{ s^\theta \mid  s \in S\}$ is pairwise disjoint. As $\c A$ is finite, $S$ must be too, so $S = \{s_1, \ldots, s_n\}$, say. By induction, for each $k$ we have that $s_1 \plusdot \ldots \plusdot s_k$ is defined and $(s_1 \plusdot \ldots \plusdot s_k)^\theta = \bigcup_{i =1}^k s_i^\theta$. Hence $(s_1 \plusdot \ldots \plusdot s_n)^\theta = \bigcup_{s\in S}s^\theta$. It is clear that for any $b_1, b_2 \in \c A$ the implication $b_1 \lesssim b_2 \implies b_1^\theta \subseteq b_2^\theta$ holds. Therefore $a^\theta$ must be a superset of each $s_i^\theta$ and must be a subset of $b^\theta$ for any upper bound $b$ of $S$. But $s_1 \plusdot \ldots \plusdot s_n$ is clearly an upper bound for $S$ so we conclude that $a^\theta = (s_1 \plusdot \ldots \plusdot s_n)^\theta = \bigcup_{s\in S}s^\theta$ as required.
\end{proof}

Finally, a word about logic. In our meta-language, that is, English, we can talk in terms of partial operations and partial algebras, which is what we have been doing so far. However, the traditional presentation of first-order logic does not include partial function symbols. Hence, in order to examine the first-order logic of our partial algebras we must view them formally as relational structures.

Let $\c A = (A, \plusdot)$ be a partial algebra. From the partial binary operation $\plusdot$ over $A$ we may define a ternary relation $J$ over $A$ by letting $J(a, b, c)$ if and only if $a\plusdot b$ is defined and equal to $c$.  Since a partial operation is (at most) single valued, we have
\begin{equation}\label{eq:singlev}
J(a, b, c)\wedge J(a, b, d)\;\rightarrow\; c=d.
\end{equation}
  Conversely, given any ternary relation $J$ over $A$ satisfying \eqref{eq:singlev}, we may define a partial operation $\plusdot$ over $A$ by letting $a\plusdot b$ be defined if and only if there exists $c$ such that $J(a, b, c)$ holds (unique, by \eqref{eq:singlev}) and when this is the case we let $a\plusdot b=c$.  The definition of $J$ from $\plusdot$ and the definition of $\plusdot$ from $J$ are clearly inverses. Similarly, if $\negdot$ is in the signature we can define a corresponding ternary relation $K$ in the same way.

To remain in the context of classical first-order logic we adopt languages that  feature neither $\plusdot$ nor $\negdot$ but have ternary relation symbols $J$ and/or $K$ as appropriate (as well as equality).  In the relational language $\lang L(J)$, we may write $\exists a\plusdot b$ as an abbreviation of the formula $\exists c J(a, b, c)$ and write $a\plusdot b=c$ in place of $J(a, b, c)$. Similarly for $\negdot$ and $K$.

\section{Axiomatisability}\label{Axiomatisability}

In this section we show there exists a first-order $\lang L(J)$-theory  that axiomatises the class $\class J$ of partial $\plusdot$-algebras with $\cupdot$-representations. Hence $\class J$, viewed as a class of $\lang L(J)$-structures, is elementary. We do the same for the class $\class K$ of partial $\negdot$-algebras with $\minusdot$-representations (as sets) and the class $\class L$ of partial $(\plusdot, \negdot)$-algebras with $(\cupdot, \minusdot)$-representations.

\begin{definition}\label{subalgebra}
If $\c A_1 \subseteq \c A_2$ are similar partial algebras and the inclusion map is a partial-algebra embedding then we say that $\c A_1$ is a \defn{partial-subalgebra} of $\c A_2$.
Let $\c A_i=(A_i, \Omega_0, \ldots)$ be partial algebras, for $i\in I$, and let $U$ be an ultrafilter over $I$.  The \defn{ultraproduct} $\Pi_{i\in I}\c A_i/U$ is defined in the normal way, noting that, for example, $[(a_i)_{i\in I}]\plusdot[(b_i)_{i\in I}]$ (where $a_i, b_i\in\c A_i$ for all $i\in I$) is defined in the ultraproduct if and only if $\set{i \in I\mid  a_i\plusdot b_i\mbox{ is defined in $\c A_i$}}\in U$. Ultrapowers and ultraroots also have their normal definitions: an \defn{ultrapower} is an ultraproduct of identical partial algebras and $\c A$ is an \defn{ultraroot} of $\c B$ if $\c B$ is an ultrapower of $\c A$.
\end{definition}

It is clear that a partial-subalgebra of $\c A$ is always a substructure of $\c A$, as  relational structures, and also that any substructure of $\c A$ is a partial algebra, that is, satisfies \eqref{eq:singlev}. However, in order for a relational substructure of $\c A$ to be a partial-subalgebra  it is necessary that it be closed under the partial operations, wherever they are defined in $\c A$.

It is almost trivial that the class of $\cupdot$-representable partial algebras is closed under partial-subalgebras. This class is not however closed under substructures. Indeed it is easy to construct a partial $\plusdot$-algebra $\c A$ with a disjoint-union representation but where an $\lang L(J)$-substructure of $\c A$ has no disjoint-union representation. We give an example now.

\begin{example}
The collection $\powerset\{1,2,3\}$ of sets forms a disjoint-union partial algebra of sets and so is trivially a $\cupdot$-representable partial $\plusdot$-algebra, if we identify $\plusdot$ with $\cupdot$.

The substructure with domain $\powerset\{1,2,3\} \setminus \{1,2,3\}$ is not $\cupdot$-representable, because $\{1\} \plusdot \{2\}, \{2\} \plusdot \{3\}$ and $\{3\} \plusdot \{1\}$ all exist, so $\{1\}, \{2\}, \{3\}$ would have to be represented by pairwise disjoint sets. But then $\{1,2\} \plusdot \{3\}$ would have to exist, which is not the case.
\end{example} 

We obtain the following corollary.

\begin{corollary}\label{notuniv}
The isomorphic closure of the class of disjoint-union partial algebras of sets is not axiomatisable by a universal first-order $\lang L(J)$-theory.
\end{corollary}

Returning to our objective of proving that the classes $\class J, \class K$ and $\class L$ are elementary, this could be achieved by showing that they are closed under ultraproducts and ultraroots. However this is not entirely straightforward, since  many of the relevant model-theoretic results are known for total operations only.  Instead, to apply these known results, we first describe a way to view an arbitrary partial algebra as a total algebra. Then, having established elementarity of the resulting class of total algebras we describe how to convert back to an axiomatisation of the partial algebras.

\begin{definition}
Let $\c A = (A, (\Omega_i)_{i < \beta})$ be a partial algebra. The \defn{totalisation} of $\c A$ is the algebra $\c A^+ =(A\cup\set{\infty}, \infty, (\Omega_i)_{i < \beta} )$, where $\infty\not\in A$ and for each $i$ the interpretation of $\Omega_i$ in $\c A^+$ agrees with the interpretation in $\c A$ whenever the latter is defined, and in all other cases returns $\infty$. The totalisation of a class $\class C$ of similar partial algebras is the class $\class C^+ = \set{\c A^+ \mid \c A\in\class C}$ of total algebras.
\end{definition}

Inversely to totalisation, suppose we have a total algebra $\c B=(B, \infty, (\Omega_i)_{i < \beta} )$ where for each $i<\beta$, if any element of the $\alpha(i)$-tuple $\bar b$ is $\infty$ then $\Omega_i(\bar b)=\infty$. Then we may define a partial algebra $\c B^-=(B\setminus\set\infty, (\Omega_i)_{i < \beta})$ where each $\Omega_i(b_1, \ldots, b_n)$ is defined in $\c B^-$ if and only if $\Omega_i(b_1, \ldots, b_n)\neq\infty$ in $\c B$, in which case it has the same value as in $\c B$. Clearly for any partial algebra $\c A$ we have $(\c A^+ )^-=\c A$ and for any total algebra $\c B$ with a suitable $\infty\in\c B$ we have $(\c B^-)^+ =\c B$. 

 In the following, we show that each of the classes $\class J^+, \class K^+$ and  $\class L^+$ is closed under both ultraproducts and subalgebras and hence is universally axiomatisable in $\lang L(\infty, \plusdot)$, $\lang L(\infty, \negdot)$ and $\lang L(\infty, \plusdot, \negdot)$ respectively.  We then give a translation from the universal formulas defining $\class J^+$ to a set of $\lang L(J)$-formulas that defines $\class J$. Similarly for the other two cases.

We will be using the notion of pseudoelementarity, and since there are various possible equivalent definitions of this, we state the one we wish to use. It can be found, for example, as \cite[Definition~9.1]{HH:book}.

\begin{definition}\label{def:pseudo}
Given an unsorted first-order language $\lang L$, a class $\class C$ of $\lang L$-structures is \defn{pseudo\-elementary} if there exist
\begin{itemize}
\item
a two-sorted first-order language $\lang L'$, with sorts $\mathit{algebra}$ and $\mathit{base}$, containing $\mathit{algebra}$-sorted copies of all symbols of $\lang L$,
\item
an $\lang L'$-theory $T$,
\end{itemize}
such that $\class C = \{M^\mathit{algebra}\restr{\lang L} \mid M \models T\}$.
\end{definition}

\begin{lemma}\label{univ}
The class $\class J^+$ is universally axiomatisable in $\lang L(\infty, \plusdot)$, 
the class $\class K^+$ is universally axiomatisable in $\lang L(\infty, \negdot)$ and
the class $\class L^+$ is universally axiomatisable in $\lang L(\infty, \plusdot, \negdot)$.
\end{lemma}

\begin{proof}
We start with $\class J^+$. By definition, $\class J^+$ is closed under isomorphism. We first show that $\class J^+ $ is pseudoelementary, hence also closed under ultraproducts. 

Consider a two-sorted language, with an algebra sort and a base sort. The signature consists of a binary operation $\plusdot$ on the algebra sort, an algebra-sorted constant $\infty$ and a binary predicate $\in$, written infix, of type $ \mathit{base}\times \mathit{algebra}$. Consider the formulas
\begin{align*}
a\plusdot\infty = \infty \plusdot a = \infty\\
(a\neq b)\wedge (a\neq\infty)\wedge (b\neq \infty)&\rightarrow \exists x ((x \in a\wedge x \not\in b)\vee(x \not\in a\wedge x \in b))\\
 (a\neq\infty)\wedge (b\neq\infty)&\rightarrow ((a\plusdot b=\infty)\biimp\exists x (x \in a\wedge x \in b))\\ 
(a\plusdot b\neq \infty)&\rightarrow ((x \in a \plusdot b) \biimp (x \in a \vee x \in b))
\end{align*}
where $a, b, c$ are algebra-sorted variables and $x$ is a base-sorted variable.  

These formulas merely state that the base-sorted elements form the base of a representation of the non-$\infty$ elements of algebra sort and that $\infty$ behaves as it should for an algebra in $\class J^+$.
Hence $\class J^+ $ is the class of $(\infty, \plusdot )$-reducts  of restrictions of models of the formulas to algebra-sorted elements, that is, $\class J^+$ is pseudoelementary. Hence $\class J^+$ is closed under ultraproducts.

Since the only function symbol, $\plusdot $, in our defining formulas is already in ${\lang L(\infty, \plusdot)}$ and there is no quantification of algebra-sorted variables, $\class J^+$ is closed under substructures.  A consequence of this is that $\class J^+$ is closed under ultraroots, by the simple observation that the diagonal map embeds any ultraroot into its ultrapower.

We now know that $\class J^+$ is closed under isomorphism, ultraproducts and ultraroots. This is a well-known algebraic characterisation of elementarity (for example see \cite[Theorem 6.1.16]{ChK90}). Then as $\class J^+$ is elementary and closed under substructures it is universally axiomatisable, by the \Los-Tarski preservation theorem.

For $\class K^+$ and $\class L^+$ the same line of reasoning applies. Each is by definition closed under isomorphism. For $\class K^+$ we show pseudoelementarity and closure under substructures using the formulas
\begin{align*}
a\negdot\infty = \infty \negdot a = \infty\\
(a\neq b)\wedge (a\neq\infty)\wedge (b\neq \infty)&\rightarrow \exists x ((x \in a\wedge x \not\in b)\vee(x \not\in a\wedge x \in b))\\
 (a\neq\infty)\wedge (b\neq\infty)&\rightarrow ((a\negdot b=\infty)\biimp\exists x (x \not\in a\wedge x \in b))\\ 
(a\negdot b\neq \infty)&\rightarrow ((x \in a \negdot b) \biimp (x \in a \wedge x \not\in b))
\end{align*}
and for $\class L^+$ we do the same using the union of the formulas for $\class J^+$ and the formulas for $\class K^+$.
\end{proof}

\begin{proposition}\label{lem:elementary}
Let $\class C$ be a class of partial algebras of the signature $(\Omega_i)_{i < \beta}$, which we view as relational structures over the signature $(R_i)_{i < \beta}$, where for each $i$ the arity of $R_i$ is one greater than that of $\Omega_i$. Suppose  $\class C^+$ is universally axiomatisable in the language $\lang L(\infty, (\Omega_i)_{i < \beta})$. Then $\class C$ is axiomatisable in the language $\lang L((R_i)_{i < \beta})$.
\end{proposition}

\begin{proof}
Let $\Sigma^+ $ be a universal axiomatisation of $\class C^+$ in the language $\lang L(\infty, (\Omega_i)_{i < \beta})$.   Since it is the validity of all the formulas in $\Sigma^+ $ that defines $\class C^+ $ we may assume that each axiom in $\Sigma^+ $ is quantifier free.  We define a translation $^-$ from $\lang L(\infty, (\Omega_i)_{i < \beta})$ to $\lang L((R_i)_{i < \beta})$ such that
\begin{equation}\label{eq:tr} \c A^+ \models \psi\iff \c A\models  \psi^-\end{equation}
for any \emph{nonempty} partial algebra $\c A$ of the signature $(\Omega_i)_{i < \beta}$ and any quantifier-free $\lang L(\infty, (\Omega_i)_{i < \beta})$-formula $\psi$.

    Let $V(\psi)$ be the finite set of variables occurring in $\psi$ and let $S(\psi)$ be the set of subterms of $\psi$. We may also write $V(t)$ and $S(t)$ to denote the set of all variables and subterms of the term $t$.  For any  assignment $\rho:V(\psi)\rightarrow \c A^+ $ and $t\in S(\psi)$ let $[t]^\rho$ denote the evaluation of $t$ under $\rho$ in $\c A^+ $. Let $v$ be any injective mapping from $S(\psi)$ to our set of first-order variables, mapping the term $s\in S(\psi)$ to the variable $v_s$ and satisfying $v_u=u$ for  all $u\in V(\psi)$. Let $V(\psi)^*=\set{v_t\mid t\in S(\psi)}$, so $V(\psi)^*\supseteq V(\psi)$. A \defn{grounded} subset $D\subseteq V(\psi)^*$ satisfies
\begin{itemize}\larger
\item $v_\infty\not\in D,$
\item $ v_{\Omega_i(t_1, \ldots, t_n)}\in D \implies v_{t_1},\ldots, v_{t_n}\in D,$ {\normalsize for any}  $\Omega_i(t_1, \ldots, t_n)\in S(\psi)\mbox{.}$
\end{itemize}
Informally, each grounded $D$ determines a partition of the subterms into `defined' terms $t$ (when $v_t\in D$) and `undefined' terms $s$ (when $v_s\in V(\psi)^*\setminus D$), in a way that is consistent with the structure of the terms.

For any subset $D\subseteq V(\psi)^*$ define
\begin{align*} \varphi(D)=\bigwedge_{\mathclap{\substack{v_{t_1}, \ldots, v_{t_n}\in D, \\ v_{\Omega_i(t_1, \ldots, t_n)}\in D}}}\;R_i(v_{t_1},\ldots, v_{t_n},  v_{\Omega_i(t_1, \ldots, t_n)})\quad\quad\wedge\quad\quad\bigwedge_{\mathclap{\substack{v_{t_1}, \ldots, v_{t_n}\in D,\\ v_{\Omega_i(t_1, \ldots, t_n)}\not\in D}}}\;\forall w\neg R_i(v_{t_1},\ldots, v_{t_n},  w)
\end{align*}
where $w$ is a new variable.
For any equation $s=t$ occurring in $\psi$ define
\[ (s=t)_D\;\mbox{ by }\; \left\{\begin{array}{ll}v_s=v_t&\mbox{if } v_s, v_t\in D\\
\top&\mbox{if }v_s, v_t\not\in D\\
\bot&\mbox{otherwise}
\end{array}
\right.\] 
and then let $\psi_D$ be obtained from $\psi$ by replacing each equation $s=t$ by $(s=t)_D$. Translate $\psi$ to the $\lang L(J)$-formula 
\[\psi^-=\bigwedge_{\mathclap{\substack{D\,\subseteq\, V(\psi)^* \\ D \text{ grounded}}}}\;( \varphi(D) \rightarrow \psi_{D})\mbox{.}\]

We must prove \eqref{eq:tr}.
Suppose $\psi$ is not valid in $\c A^+ $, say $\rho:V(\psi)\rightarrow\c A^+ $ is an assignment such that $\c A^+ ,\rho\not\models\psi$.  Let $\rho^*:V(\psi)^*\rightarrow\c A$ satisfy $\rho^*(v_t)=[t]^\rho$ for any $t\in S(\psi)$ provided $[t]^\rho\neq\infty$ (else $\rho^*(v_t)\in\c A$ is arbitrary---possible since $\c A$ is nonempty) and let $D=\set{v_t\in V(\psi)^*\mid  [t]^\rho\neq\infty}$.  Then $D$ is grounded,  $\c A, \rho^*\models\varphi(D)$ and  by formula induction  we get $\c  A^+ , \rho\models\chi\iff \c A, \rho^*\models\chi_D$ for any subformula $\chi$ of $\psi$. So $\c A, \rho^*\not\models \varphi(D)\rightarrow\psi_D$ and therefore $\psi^-$ is not valid in $\c A$.

Conversely, suppose $\psi^-$ is not valid in $\c A$, so there is a grounded subset $D\subseteq V(\psi)^*$ and a variable assignment $\mu:V(\psi^-)\rightarrow\c A$ such that  $\c A, \mu\models\varphi(D)\wedge\neg\psi_D$. As $\c A$ is nonempty we may extend $\mu$ to an assignment $\lambda:V(\psi)^*\rightarrow\c A$ and we will have  $\c A, \lambda\models\varphi(D)\wedge\neg\psi_D$. Define $\lambda^+ :V(\psi)\rightarrow\c A^+ $ by
\[ \lambda^+ (u)=\left\{\begin{array}{ll}\lambda(u)&u\in D\\ \infty &u\in V(\psi)\setminus D.\end{array}\right.\] 
Since $\c A, \lambda\models\varphi(D)$ and $\lambda^+ $ agrees with $\lambda$ over $D\cap V(\psi)$ we have 
\begin{equation}\label{eq:algplus }
[t]^{\lambda^+ }=\left\{\begin{array}{ll}
\lambda(v_t)&v_t\in D\\
\infty&v_t\not\in D.
\end{array}\right.
\end{equation}  We claim that
\begin{equation}\label{eq:s=t}
\c A^+ , \lambda^+ \models \chi\iff\c A, \lambda\models \chi_D
\end{equation}
for any subformula $\chi$ of $\psi$.  For the base case let $\chi$ be an equation $(s=t)$.
If $v_s, v_t\in D$ then $\chi_D$ is $v_s=v_t$ and \eqref{eq:s=t} holds, by \eqref{eq:algplus }.  If $v_s\in D$ but $v_t\not\in D$ then $\chi_D$ is $\bot$ and  $\infty=[t]^{\lambda^+ }\neq[s]^{\lambda^+}\in\c A$ so both sides of \eqref{eq:s=t} are false. The case where $v_t\in D$ but $v_s\not\in D$ is similar.  Finally, if $v_s, v_t\not\in D$ then $\chi_D$ is $\top$ and $ [s]^{\lambda^+}= [t]^{\lambda^+}=\infty$ so both sides of \eqref{eq:s=t} are true.    Now \eqref{eq:s=t} follows for all subformulas $\chi$ of $\psi$, by a simple structural induction.  Since $\c A, \lambda\not\models\psi_D$ we deduce that $\c A^+ , \lambda^+ \not\models\psi$, so $\psi$ is not valid in $\c A^+ $.  This completes the proof of \eqref{eq:tr}.

If $\c A$ is nonempty, we have 
\[\c A \in \class C \iff \c A^+ \in \class C^+ \iff \c A^+ \models \Sigma^+ \iff \c A \models \set{\psi^-\mid \psi\in\Sigma^+ }.\]
So if the empty partial algebra is in $\class C$ then $\set{\forall x \psi^-\mid \psi\in\Sigma^+ }$ is an axiomatisation of $\class C$. If the empty partial algebra is \emph{not} in $\class C$ then  $\set{\exists x \top \wedge\psi^-\mid \psi\in\Sigma^+ }$ is an axiomatisation of $\class C$.
\end{proof}

\begin{theorem}
Let $\sigma$ be any one of the signatures $(\cupdot)$, $(\minusdot)$ or $(\cupdot, \minusdot)$. The class of partial algebras $\sigma$-representable as sets, viewed as a class of relational structures, is elementary.
\end{theorem}

\begin{proof}
The classes in question are $\class J$, $\class K$ and $\class L$.  \Cref{univ} tells us that each class satisfies the condition for \Cref{lem:elementary} to apply. Hence each class is axiomatisable in the appropriate relational language.
\end{proof}

We can now easily establish elementarity  in all cases without composition.

\begin{corollary}\label{sets}
Let $\sigma$ be any signature whose symbols are a subset of $\set{\cupdot, \minusdot, \emptyset}$.  The class of partial algebras that are $\sigma$-representable by sets is elementary.
\end{corollary}

\begin{proof}
The previous theorem gives us the result for the three signatures $(\cupdot)$, $(\minusdot)$ and $(\cupdot, \minusdot)$. Then as we noted in \Cref{add0}, axiomatisations for these signatures yield axiomatisations for the signatures $(\cupdot, \emptyset)$, $(\minusdot, \emptyset)$ and $(\cupdot, \minusdot, \emptyset)$ with the addition of a single extra axiom, either $J(0,0,0)$ or $K(0,0,0)$. The remaining cases, the empty signature and the signature $(\emptyset)$, trivially are axiomatised by the empty theory.
\end{proof}

\begin{corollary}
Let $\sigma$ be any signature whose symbols are a subset of $\set{\dcupdot, \minusdot, \emptyset}$.  The class of partial algebras that are $\sigma$-representable by partial functions is elementary.
\end{corollary}

\begin{proof}
By  \Cref{prop}\eqref{first} these representation classes are the same as those in \Cref{sets}.
\end{proof}

\section{A Recursive Axiomatisation via Games}\label{Recursive}

In this section we describe a recursive axiomatisation of the class of $\cupdot$-repre\-sentable partial algebras. This axiomatisation can be understood quite simply, as a sequence of statements about a particular two-player game. The efficacy of this approach using games relies on our prior knowledge, obtained in the previous section, that the class in question is elementary. The reader should note that everything in this section can be adapted quite easily to $\minusdot$-representability by sets and $(\cupdot, \minusdot)$-representability by sets.

Fix some partial $\plusdot$-algebra $\c A$. The following definition and lemma are the motivation behind our two-player game.

\begin{definition}
We call a subset $U$ of $\c A$
\begin{itemize}
\item
\defn{$\plusdot$-prime} if  $a\plusdot b\in U$ implies either $a\in U$ or $b\in U$,
\item
\defn{bi-closed} if the two conditions $a\in U$ or $b\in U$ and  $a\plusdot b$  defined,  together imply $a\plusdot b\in U$,
\item
\defn{pairwise incombinable} if  $a, b\in U$ implies $a\plusdot b$ is undefined.
\end{itemize} 

\end{definition}
\begin{lemma}\label{lem:sets}
Let $\mathcal F(\c A)$ be the set of all $\plusdot$-prime, bi-closed, pairwise-incombinable subsets of $\c A$. Then $\c A$ has a disjoint-union representation if and only if there is a $B\subseteq \mathcal F(\c A)$ such that
\begin{enumerate}[label=(\roman*)]
\item\label{neq} for all $a\neq b\in\c A$ there is $U\in B$ such that either $a\in U$ and $b\not\in U$ or $b\in U$ and $a\not\in U$,
\item\label{defined} for all $a, b\in\c A$ if $a\plusdot b$ is undefined then there is $U\in B$ such that $a, b\in U$.
\end{enumerate}
\end{lemma}
\begin{proof}
For the left-to-right implication, if  $\theta$ is  a disjoint-union representation of $\c A$ on a base set $X$ then for each $x\in X$ let $U(x)=\set{a\in \c A\mid x\in a^\theta}$ and let $B=\set{U(x)\mid x\in X}$.   It is easy to see that  $U(x)$ is a $\plusdot$-prime, bi-closed, pairwise-incombinable set, for all $x\in X$, and that $B$ includes all  elements required by \ref{neq} and \ref{defined} of this lemma.  

Conversely,  assuming that $B\subseteq \mathcal F(\c A)$ has the required elements we can define a representation $\theta$ of $\c A$ by $a^\theta=\set{U\in B\mid a\in U}$. Condition \ref{neq} ensures that $\theta$ is faithful, that is, distinguishes distinct elements of $\c A$. Condition \ref{defined} ensures $a^\theta$ and $b^\theta$ are disjoint only if $a \plusdot b$ is defined. The pairwise incombinable condition on each $U \in B$ ensures $a \plusdot b$ is defined only if $a^\theta$ and $b^\theta$ are disjoint. The $\plusdot$-prime and bi-closed conditions on elements of $B$ ensure that when $a \plusdot b$ is defined, $(a \plusdot b)^\theta = a^\theta \cup b^\theta$.
\end{proof}

 We define a two player game $\Gamma_n$ over $\c A$ with $n\leq\omega$ rounds, played by players $\forall$ and $\exists$.   A position $(Y, N)$ consists of two finite subsets $Y$ and $N$ of $\c A$.  It might help to think of $Y$ as a finite set of sets  such that some given point belongs to each of them and $N$ is a finite set of sets such that  the same point belongs to none of them.

  In the initial round (round $0$) $\forall$ either
\begin{enumerate}[label=(\roman*)]
\item\label{faithful}
picks $a\neq b\in \c A$, or
\item\label{intersect}
picks $a, b\in\c A$ such that $a\plusdot b$ is undefined.
\end{enumerate}
  In the former case $\exists$ responds with an initial position, either $(\set a,  \set b)$ or $(\set b,  \set a)$, at her choice.  In the latter case she must  respond with the initial position $(\set{a, b},\emptyset) $.

In all later rounds, if the position is $(Y, N)$  then $\forall$ either
\begin{enumerate}[label=(\alph*)]
\item\label{down}
picks $a, b\in\c A$ such that $a\plusdot b$ is defined and belongs to $Y$, or
\item\label{b}
picks $a\in Y$ and $b\in\c A$ such that $a\plusdot b$ is defined, or
\item\label{c}
picks $a\in \c A$ and $b\in Y$ such that $a\plusdot b$ is defined.
\end{enumerate}
   In case \ref{down} player $\exists$ responds with either $(Y\cup\set a, N)$ or $(Y\cup \set b, N)$, in cases \ref{b} and \ref{c} she must respond with the position $(Y\cup\set{a\plusdot b}, N)$.  Observe that $N$ never changes as the game proceeds, it is either a singleton or empty.

A position $(Y, N)$ is a win for $\forall$ if either
\begin{enumerate}
\item\label{one}
$Y\cap N\neq \emptyset$, or 
\item\label{two}
there are $a, b\in Y$ such that $a\plusdot b$ is defined.
\end{enumerate}
 Player  $\forall$ wins a play of $\Gamma_n$ if he wins in some round $0 \leq i<n$, else $\exists$ wins the play of the game.

The game $\Gamma_n(Y, N)$ is similar (where $Y, N$ are  finite subsets of $\c A$), but the initial round is omitted and play begins from the position $(Y, N)$.

\begin{lemma}\label{lem:game}If $\c A$ is representable then $\exists$ has a winning strategy for $\Gamma_\omega$.  If $\c A$ is countable and $\exists$ has a \ws\ for $\Gamma_\omega$ then $\c A$ has a representation on a base of size at most $2|\c A|^2$.
\end{lemma}

\begin{proof}
First suppose $\c A$ has a representation, $\theta$ say.  By \Cref{lem:sets} there is a set $B$ of $\plusdot$-prime, bi-closed, pairwise-incombinable subsets of $\c A$ such that  \ref{neq} for all $a\neq b\in\c A$ there is $U\in B$ such that either $a\in U,\;b\not\in U$ or $b\in U,\;a\not\in U$ and \ref{defined} whenever $a\plusdot b$ is undefined there is $U\in B$ with $a, b\in U$.
We describe a winning strategy for $\exists$.  In response to any initial $\forall$-move she will select a suitable $U\in B$ and play an initial position $(Y, N)$ such that 
\begin{equation}\label{eq:YN} Y\subseteq U \mbox{ and } N\cap U=\emptyset.
\end{equation} 
and the remainder of her strategy will be to preserve this condition throughout the play.

In the initial round there are two possibilities.
\begin{enumerate}[label=(\roman*)]
\item
If $\forall$ plays $a\neq b\in\c A$ then there is a $U\in B$ with either $a\in U,\; b\not\in U$ or $b\in U,\; a\not\in U$. In the former case $\exists$ plays an initial position $(\set a, \set b)$ and in the latter case she plays $(\set b, \set a)$. 
\item
If $\forall$ plays $(a, b)$ where $a\plusdot b$ is undefined,  there is $U\in B$ where $a, b\in U$ and $\exists$ selects such a $U$ and plays $(\set{a, b}, \emptyset)$.  
\end{enumerate}
In each case, \eqref{eq:YN} holds.  

In a subsequent round, if the current position $(Y, N)$ satisfies \eqref{eq:YN} and $\forall$ plays $a, b$ where $a\plusdot b\in Y$ is defined then since $U$ is $\plusdot$-prime either $Y\cup\set a\subseteq U$ or $Y\cup\set b\subseteq U$, so $\exists$ may play either $(Y\cup\set a, N)$ or $(Y\cup\set b, N)$, as appropriate, preserving \eqref{eq:YN}.  Similarly, if $\forall$ plays $a, b$ where $a\in Y$ and $a\plusdot b$ is defined (or $b\in Y$ and $a\plusdot b$ is defined), then since $U$ is bi-closed we have $a\plusdot b\in U$ so $\exists$ plays $(Y\cup\set{a\plusdot b}, N)$, preserving condition \eqref{eq:YN}.    This condition suffices to prove that $\exists$ does not lose in any round of the play.
  
  Conversely, suppose $\c A$ is countable and $\exists$ has a \ws\  for $\Gamma_\omega$. Then for each $a\neq b\in \c A$ let $S_{a, b}=\bigcup_{i<\omega} Y_i$, where $(Y_0, N), (Y_1, N), \ldots$ is a play of $\Gamma_\omega$ in which  $\forall$ plays  the type \ref{faithful} move  $(a, b)$ initially (so $N$ is a singleton). For each $a, b\in\c A$ where $a\plusdot b$ is undefined let $T_{a, b} = \bigcup_{i<\omega}Y_i$ be the limit of a play in which $\forall$ plays the type \ref{intersect} move $(a, b)$ initially (so $N$ is empty). In each case we suppose---here is where we use the hypothesis that $\c A$ is countable---that $\forall$ plays all possible moves subsequently. We also suppose that $\exists$ uses her \ws.

Each set $S_{a, b}$ (where $a\neq b$) or $T_{a, b}$ (where $a\plusdot b$ is undefined) is $\plusdot$-prime, bi-closed and pairwise incombinable, since $\forall$ plays all possible moves in a play and $\exists$ never loses.  Hence $B=\set{S_{a, b}\mid a\neq b\in\c A}\cup\set{T_{a, b}\mid a\plusdot b\mbox{ is undefined}}$ satisfies the conditions of \Cref{lem:sets}. Clearly the size of the base set $B$ is at most $2|\c A|^2$.
\end{proof}

\begin{lemma}\label{lem:fmla}
  For each $n<\omega$ there is a first-order $\lang L(J)$-formula $\rho_n$ such that $\c A\models\rho_n$ if and only if $\exists$ has a \ws\ in $\Gamma_n$.
\end{lemma}

\begin{proof}
Let $V$ and $W$ be disjoint finite sets of variables.  For each $n<\omega$ we define formulas $\mu_n(V, W)$ in such a way that for any partial  $\plusdot$-algebra $\c A$ and any variable assignment $\lambda:\mathit{vars}\rightarrow\c A$ we have 
\begin{equation}\label{eq:lambda}
\c A, \lambda\models \mu_n(V, W) \iff \exists \mbox{ has a \ws\ in }\Gamma_n(\lambda[V], \lambda[W])\mbox{.}
\end{equation}

Let 
\begin{align*}
\mu_0(V, W)&=\bigwedge_{\mathclap{v, v'\in V}}\;\neg\exists c J(v, v', c)\;\wedge\; \bigwedge_{\mathclap{v\in V,\; w\in W}}\;v\neq w
\end{align*}
where $c$ is a fresh variable. So \eqref{eq:lambda} is clear when $n=0$.
For the recursive step let
\begin{align*}
\mu_{n+1}(V, W)=\forall a, b&\; \bigg(\bigwedge_{v\in V} (J(a, b, v)\rightarrow \mu_n(V\cup\set a, W)\vee\mu_n(V\cup\set b, W))\\
& \;\wedge \bigwedge_{v\in V} ( J(a, v, b)\rightarrow\mu_n(V\cup\set{b}, W))\\
& \; \wedge\bigwedge_{v\in V} ( J(v, a, b)\rightarrow\mu_n(V\cup\set{b}, W))\bigg)
\end{align*}
where $a$ and $b$ are fresh variables. By a simple induction on $n$ we see that \eqref{eq:lambda} holds, for all $n$.
Finally, (let $\rho_0 = \top$ and) let
\[\rho_{n+1}=\forall a, b\bigg( (a= b\vee \mu_{n}(\set a, \set b)\vee\mu_{n}(\set b,\set a))\wedge (\exists c J(a, b, c)\vee\mu_{n}(\set{a, b}, \emptyset))\bigg)\]
where again $a, b$ and $c$ are fresh variables.
\end{proof}

Observe that each formula $\mu_n(V, W)$ is equivalent to a universal formula and therefore $\rho_n$, but for the clause $\exists c J(a, b, c)$,  is universal.

\begin{theorem}
The isomorphic closure of the class of disjoint-union partial algebras of sets is axiomatised by $\set{\rho_n\mid n<\omega}$.
\end{theorem}

\begin{proof}
We will use \Cref{lem:game} and \Cref{lem:fmla}, but we must be slightly careful, because we chose to present the lemmas with the assumption that the $\lang L(J)$-structure in question is a partial algebra. Hence we must check that \eqref{eq:singlev} holds before appealing to either lemma.

If an $\lang L(J)$-structure $\c A$ is isomorphic to a disjoint-union partial algebra of sets then certainly it satisfies \eqref{eq:singlev}. Then by \Cref{lem:game}, player $\exists$ has a winning strategy in the game of length $n$ for each $n<\omega$. So $\c A \models\rho_n$ by \Cref{lem:fmla}.

Conversely, if $\c A \models\set{\rho_n\mid n<\omega}$ let $\c B$ be any countable elementary substructure of $\c A$. Then $\c B \models\set{\rho_n\mid n<\omega}$. The validity of $\rho_3$ tells us that $\eqref{eq:singlev}$ holds, as we now explain. For if $J(a,b,c)$ and $J(a,b,d)$, with $c \neq d$, then from $\rho_3$ we know that either $ \mu_2(\set c, \set d)$ holds or $\mu_2(\set d, \set c)$ holds. Without loss of generality, we assume the former.   From $\mu_2(\set c, \set d)$, assigning $c$ to $v$ in the first conjunct, we deduce $\mu_1(\set{c, a}, \set d)$ or $\mu_1(\set{c, b}, \set d)$ and again we may assume the former.  From the second conjunct in $\mu_1(\set{c, a}, \set d)$ (assigning $b$ to the variable $v$ and $d$ to the variable $b$) we deduce $\mu_0(\set{c, a, d}, \set d)$, which is contradicted by the final inequality $v\neq w$, when $v$ and $w$ are both assigned $d$.

  Hence we can use \Cref{lem:fmla} and conclude that $\exists$ has a \ws\ in game $\Gamma_n$ for each $n<\omega$. Then since $\exists$ has only finitely many choices open to her in each round (actually, at most two choices), by K\"onig's tree lemma she also has a \ws\ in $\Gamma_\omega$. So by \Cref{lem:game} the partial algebra $\c B$ is isomorphic to a disjoint-union partial algebra of sets.   Since $\c A$ is elementarily equivalent to $\c B$, we deduce   $\c A$ is also isomorphic to a disjoint-union partial algebra of sets, by \Cref{lem:elementary}.
\end{proof}

\section{Non-axiomatisability}\label{Non-axiomatisability}

In this section we show that for any of the signatures $(\cupdot), (\minusdot), (\cupdot, \minusdot), (\cupdot, \emptyset), {(\minusdot, \emptyset)}$ or $(\cupdot, \minusdot, \emptyset)$ the class of  partial algebras representable by sets is not finitely axiomatisable. Hence the same is true for representability by partial functions, when $\cupdot$ is replaced by $\dcupdot$. For partial functions, we also show the same holds when we add composition to these signatures. Our strategy is to describe a set of non-representable partial algebras that has a representable ultraproduct.

Let $m$ and $n$ be sets of cardinality greater than two.   We will call a subset of $m\times n$ \defn{axial} if it has the form $\set i\times J$ (for some $i \in m,\;  J\subseteq n$) or the form $I\times\set j$ (for some $I\subseteq m,\; j \in n$).       Observe that $\emptyset\times \set  j=\set i\times  \emptyset=\emptyset$ for any $i\in m,\; j\in n$.  

Next we define a partial  $(\plusdot, 0)$-algebra $\c X(m, n)$. It has a domain consisting of all axial subsets of $m\times n$. The constant $0$ is interpreted as the empty set and $S\plusdot T$  is defined and equal to $S\cup T$ if $S$ is disjoint from $T$ and $S\cup T$ is axial, else it is undefined.    

Now for any $i\neq i'\in m$ and $j\neq j'\in n$ the set $I(i, j, i', j')=\set{a\in\c X(m, n)\mid (i, j)\in a\mbox{ or }(i', j')\in a}$ is a $\plusdot$-prime, bi-closed, pairwise-incombinable set. The collection $\set{I(i, j, i', j')\mid i\neq i'\in m,\; j\neq j'\in n}$ of such sets satisfies conditions \ref{neq} and \ref{defined} of \Cref{lem:sets} and hence the $\plusdot$-reduct of $\c X(m, n)$ is $\cupdot$-representable. The formula $0 \plusdot 0 = 0$ is also satisfied, so $\c X(m, n)$ is $(\cupdot, \emptyset)$-representable.

\begin{definition}
Given a partial algebra $\c A=(A, (\Omega_i)_{i < \beta})$, a \defn{partial-algebra congruence} on $\c A$ is an equivalence relation $\sim$ with the property that for each $i$ and every $a_1, \ldots, a_{\alpha(i)}, b_1, \ldots, b_{\alpha(i)} \in \c A$, if $a_1\sim b_1$, \ldots, $a_{\alpha(i)}\sim b_{\alpha(i)}$ then $\Omega_i(a_1, \ldots, a_{\alpha(i)})$ is defined if and only if $\Omega_i(b_1, \ldots, b_{\alpha(i)})$ is defined and when these are defined $\Omega_i(a_1, \ldots, a_{\alpha(i)}) \sim \Omega_i(b_1, \ldots, b_{\alpha(i)})$.
\end{definition}

Note our condition for being a partial-algebra congruence is strictly stronger than that obtained by viewing a partial algebra as a relational structure and then using the familiar definition of a congruence---for signatures with no function symbols a congruence relation is merely an equivalence relation. Our definition of a partial-algebra congruence takes the `algebraic' rather than `relational' view of the structure.
 If a congruence $\sim$ on a partial algebra $\c A$ is a partial-algebra congruence, that is sufficient for the quotient $\c A /{\sim}$ (with its usual meaning on the relational view of $\c A$) to itself be a partial algebra.

\begin{definition}
Given a partial algebra $\c A=(A, (\Omega_i)_{i < \beta})$ and a partial-algebra congruence $\sim$ on $\c A$, the \defn{partial-algebra quotient} of $\c A$ by $\sim$, written $\c A/{\sim}$, is the partial algebra of the signature $(\Omega_i)_{i < \beta}$ with domain the set of $\sim$-equivalence classes and well-defined partial operations given by: $\Omega_i([a_1], \ldots, [a_{\alpha(i)}])$= $[\Omega_i(a_1, \ldots, a_{\alpha(i)})]$ if $\Omega_i(a_1, \ldots, a_{\alpha(i)})$ is defined, else $\Omega_i([a_1], \ldots, [a_{\alpha(i)}])$ is undefined.
\end{definition}

Given  a partial-algebra congruence $\sim$ on $\c A$, the partial-algebra quotient by $\sim$ is the same structure  as the quotient by $\sim$, which is why we reuse the notation $\c A/{\sim}$. All the expected relationships between partial-algebra homomorphisms, partial-algebra congruences and partial-algebra quotients hold.

Returning to our task, we define a binary relation $\sim$ over $\c X(m, n)$ as the smallest equivalence relation such that
\begin{align*}
\set i\times n&\sim m \times \set j\\
\set i \times (n \setminus\{j\})&\sim (m \setminus \{i\}) \times \set j
\end{align*}
for all $i \in m,\; j \in n$.  The equivalence class of $\set i\times n$ (for any choice of $i \in m$) is denoted $1$
and the equivalence class of $\set i \times (n \setminus \{j\})$ is denoted $\overline{(i, j)}$, for each $i \in m,\; j \in n$. All other equivalence classes are singletons, either $\set{\set i\times J}$ for some $i \in m,\; J\subsetneq n$  or  $\set{I \times \set j}$ for some $ I\subsetneq m, \;j\in n$.   We show next that $\sim$ is a partial-algebra congruence. Clearly $\plusdot$ is commutative in the sense that $S \plusdot T$ is defined if and only if $T \plusdot S$ is defined and then they are equal. Hence it suffices to show, for any $S\sim S'$, that $S\cap T=\emptyset$ and $S\cup T$ is axial if and only if $S'\cap T=\emptyset$ and $S'\cup T$ is axial, and if these statements are true then $S\cup T\sim S'\cup T$. Further, by symmetry, it suffices to prove only one direction of this biconditional.

Suppose then that $S\sim S'$, that $S\cap T=\emptyset$ and that $S\cup T$ is axial. 
 We may assume  $S \neq S'$, so without loss of generality there are two cases to consider: the case $S = \set i\times n$ and the case $S = \set i\times (n \setminus \{j\})$ and $S' = (m \setminus\{i\}) \times \set j$. In the first case, since $S \cup T$ is axial and $|n| >1$ we know $T$ must be a subset of $S$. But $T$ is also disjoint from $S$, hence $T$ is empty. Then it is clear that $S'\cap T=\emptyset$ and $S'\cup T$ is axial and that $S \cup T \sim S' \cup T$. In the second case, since $S \cup T$ is axial and $|n| >2$ we know $T$ must be a subset of $\set i\times n$. But $T$ is also disjoint from $S$ and so $T$ is either $\set{(i, j)}$ or $\emptyset$. Either way, it is clear that $S'\cap T=\emptyset$ and $S'\cup T$ is axial and that $S \cup T \sim S' \cup T$.

Now define a partial $(\plusdot, 0)$-algebra $\c A(m, n)$  as the partial-algebra quotient \linebreak$\c X(m, n)/{\sim}$.      Since the elements of $\c A(m, n)$ are $\sim$-equivalence classes  and these are typically singletons, we will suppress the $[\,\cdot\,]$ notation and let the axial set $S$ denote the equivalence class of $S$, taking care to identify $\sim$-equivalent axial sets.

For the following lemma, recall the notion of $\lesssim$-complete representability given in \Cref{completely}.

\begin{lemma}\label{lem:mn}
For any sets $m$ and $n$ of cardinality greater than two, the partial algebra $\c A(m, n)$ is $\lesssim$-completely $(\cupdot, 0)$-representable if and only if $|m|=|n|$.
\end{lemma}

\begin{proof}
For the left-to-right implication let $\theta$ be a $\lesssim$-complete representation of $\c A(m, n)$ over the base $X$. The set $1^\theta$ must be nonempty, because $1 \plusdot 1$ is undefined. Fix some $x\in 1^\theta$ and define a subset  $R$ of $m\times n$ by letting $(i, j)\in R\iff x\in \set{(i, j)}^\theta$ for $i \in m,\; j \in n$.  
 For each $i \in m$, since $1$ is the supremum of $\set{\set{(i, j)}\mid  j \in n}$ and $\theta$ is $\lesssim$-complete, there is $j \in n$ such that $x\in \set{(i, j)}^\theta$ and hence $(i, j)\in R$.  Dually, for any $j \in n$, since $1$ is the supremum of $\set{\set{(i, j)}\mid i \in m}$ there is $i \in m$ such that $(i, j)\in R$.  We cannot have $(i, j), (i', j)\in R$, for distinct $i, i' \in m$ since $\theta$ is a representation and $\set{(i, j)}\plusdot\set{ (i', j)}$ is defined.  Similarly, for distinct $j, j' \in n$ we cannot have $(i, j), (i, j')\in R$.   Hence  $R$ is a bijection from $m$ onto $n$. We deduce that $|m|=|n|$.

For the right-to-left implication suppose $|m|=|n|$. It suffices to describe a $\lesssim$-complete representation of $\c A(n, n)$.

The base of the representation is the set $P_n$ of all permutations on $n$.  If $S$ is any axial set it has the form $\set i\times J$ for some $i \in n,\; J\subseteq n$ or the form $I \times \set j$ for some $I\subseteq n,\; j \in n$.  Define a representation $\theta$ over $P_n$ by letting $(\set i\times J)^\theta$ be the set of all permutations $\sigma\in P_n$ such that $\sigma(i)\in J$ and $(I \times \set j)^\theta$ be the set of all permutations $\sigma\in P_n$ such that $\sigma^{-1}(j)\in I$.   Observe this is well defined, since firstly if an axial set is both of the form $\set i\times J$ and of the form $I \times \set j$ then the definitions agree, and secondly it is easily seen that $\sim$-equivalent axial sets are assigned the same set of permutations. 

We now show that $\theta$ is a $(\cupdot, 0)$-representation. To see that $\theta$ is faithful we show that $\sim$-inequivalent axial sets are represented as distinct sets of permutations. We may assume the axial sets are not in the equivalence class 1, since $1^\theta = P_n$ and all axial sets not in 1 are clearly assigned proper subsets of $P_n$. Similarly, we may assume the axial sets are not the empty set.

First suppose we have two inequivalent \emph{vertical} sets $\set i\times J$ and $\set {i'}\times J'$. If $i = i'$ there must be a $j$ in the symmetric difference of $J$ and $J'$. Then any permutation with $i \mapsto j$ witnesses the distinction between $(\set i\times J)^\theta$ and $(\set {i'}\times J')^\theta$. Otherwise $i \neq i'$, and if we can choose  $j \neq j'$ with $j \in J$ and  $j' \not\in J'$ then any permutation with $i \mapsto j$ and $i' \mapsto j'$ belongs to $(\set i\times J)^\theta\setminus (\set {i'}\times J')^\theta$. Since we assumed our axial sets are neither $\emptyset$ nor in $1$ we can do this unless  $J$ and $n\setminus J'$ are the same singleton set, $\set {j_0}$ say. But then for any distinct $j, j'\in n\setminus\set {j_0}$ we have $j\not\in J$ and $j'\in J'$ so any permutation with $i\mapsto j,\; i'\mapsto j'$  belongs to $(\set {i'}\times J')^\theta\setminus(\set i\times J)^\theta$. Hence $\theta$ always distinguishes inequivalent vertical sets. If we have two inequivalent \emph{horizontal} sets $I \times \set j$ and $I' \times \set j'$ then the argument is similar. 

Lastly, suppose we have inequivalent sets $\set i\times J$ and $I \times \set j$. If we can choose a  $k \in J$ not equal to $j$ and an $l \not\in I$ not equal to $i$ then there exist permutations with $i \mapsto k$ and $l \mapsto j$ and any such permutation belongs to  $(\set i \times J)^\theta\setminus (I \times \set j)^\theta$. We can do this unless either $J = \{j\}$, in which case we have two horizontal sets, which we have already considered, or $I = n \setminus \{i\}$. By a symmetrical argument, we can witness the distinction unless $J = n \setminus \{j\}$. Hence $(\set i \times J)^\theta \neq (I \times \set j)^\theta$ unless $\set i\times J = \set i\times (n \setminus \{j\})$ and $I \times \set j = (n \setminus \{i\}) \times \set j$, contradicting the assumed inequivalence of $\set i\times J$ and $I \times \set j$. This completes the argument that $\theta$ is faithful.

It is clear that $\theta$ correctly represents $0$ as $\emptyset$. Now to see that $\theta$ is a $(\cupdot, \emptyset)$-representation it remains to show that $\theta$ represents $\plusdot$ correctly as $\cupdot$. If $S \plusdot T$ is defined then we may assume $S = \set i\times J_1$ and $T = \set i\times J_2$ for some disjoint $J_1$ and $J_2$, since the case where $S \plusdot T$ is a horizontal set is similar. Then it is clear from the definition of $\theta$ that $S^\theta$ and $T^\theta$ are disjoint and so $S^\theta \cupdot T^\theta$ is defined and that $(S \plusdot T)^\theta = S^\theta \cupdot T^\theta$. If $S \plusdot T$ is undefined then either there is some $(i, j) \in S \cap T$, in which case $S^\theta$ and $T^\theta$ clearly are non-disjoint, or $S \cup T$ is not axial, in which case there are $i \neq i'$ and $j \neq j'$ with $(i, j) \in S$ and $(i', j') \in T$. In the second case, any permutation with $i \mapsto j$ and $i' \mapsto j'$ witnesses that $S^\theta$ and $T^\theta$ are non-disjoint. Hence when $S \plusdot T$ is undefined, $S^\theta \cupdot T^\theta$ is undefined. This completes the proof that $\theta$ is a $(\cupdot, \emptyset)$-representation.

Finally we show that $\theta$ is $\lesssim$-complete. Let $\gamma$ be a pairwise-combinable subset of $\c A(n, n)$. If $\gamma$ has supremum $\set i\times J$ for some $J$ with $|n \setminus J| \geq 2$ then for all $S\in \gamma$ since the supremum is an upper bound and by the definition of $\lesssim$, either $S = \set i\times J$ or there is $T$ such that $S\plusdot T \sim \set i\times J$.  It follows that each $S\in \gamma$  has the form $\set i\times J_S$ for some $J_S\subseteq J$ and since the $\set i\times J$ is the least upper bound we have $J=\bigcup_{S\in \gamma} J_S$.    Then for any $\sigma\in P_n$ we have 
\begin{align*}
\sigma\in (\set i\times J)^\theta&\iff \sigma(i)\in J\\
&\iff \sigma(i)\in J_S \;\mbox{for some $S\in \gamma$}\\
&\iff \sigma\in (\set i\times J_S)^\theta \;\mbox{for some $S\in \gamma$}\\
&\iff \sigma\in \bigcup_{S\in \gamma}(\set i\times J_S)^\theta=\bigcup_{S\in \gamma}S^\theta\mbox{.}
\end{align*}
Similarly if the supremum of $\gamma$ is $I \times \set j$ for some $I$ with $|m \setminus I| \geq 2$, then $(I \times \set j)^\theta =\bigcup_{S\in \gamma} S^\theta$. 

If the supremum of $\gamma$ is $\overline{(i, j)}$ then either $\gamma = \{\overline{(i, j)}\}$, so the proof of the required equality is trivial, or, because $\gamma$ is pairwise combinable,  each $S\in \gamma$  has the form $\set i\times J_S$ or each $S\in \gamma$  has the form $I_S \times \set j$ in which cases the proof is similar to above. If the supremum of $\gamma$ is $1$, then either $\gamma = \{1\}$ or $\gamma = \{\set{(i, j)}, \overline{(i, j)}\}$ for some $i, j$, or  each $S\in \gamma$  has the form $\set i\times J_S$, or each $S\in \gamma$  has the form $I_S \times \set j$. In every case the required equality is seen to hold. So $\theta$ is a $\lesssim$-complete representation.
\end{proof}

\begin{remark}
We have seen that $\c X(3,4)$ has a $(\cupdot, \emptyset)$-representation, but,  by Lemma \ref{lem:mn} and \Cref{prop:complete}, the partial algebra $\c A(3, 4)=\c X(3, 4)/{\sim}$ does not.  Since the latter is a partial-algebra homomorphic image of the former we see that the class of $(\cupdot, \emptyset)$-representable partial algebras is not closed under partial-algebra homomorphic images, in contrast to the corresponding result for algebras representable as fields of sets, that is, boolean algebras.
\end{remark}
We now have a source of non-representable partial algebras with which to prove our first non-axiomatisability result.

\begin{theorem}\label{thm:nfa}
The class of $(\cupdot, \emptyset)$-representable partial algebras  is not finitely axiomatisable.
\end{theorem}

\begin{proof}
Write $\nu$ for $\omega \setminus \{0,1,2\}$ and let $m \in \nu$.   By \Cref{lem:mn} the partial algebra $\c A(m, m+1)$ has no $\lesssim$-complete $(\cupdot, \emptyset)$-representation. Since this partial algebra is finite, it follows, by \Cref{prop:complete}, that it has no $(\cupdot, \emptyset)$-representation.

Let $U$ be a non-principal ultrafilter over $\nu$.  We claim that the ultraproduct $\Pi_{m\in\nu}\c A(m, m+1)/U$ is isomorphic to a partial-subalgebra of $\c A(\Pi_{m\in\nu}m/U,\; \Pi_{m\in\nu}\allowbreak(m+1)/U)$. Note that every element of $\Pi_{m\in\nu}\c A(m, m+1)/U$   is the equivalence class of a sequence of vertical  sets $[(\set {i_m}\times J_m)_{m \in \nu}]$ where $i_m\in m$ and $J_m\subseteq m+1$ for each $m \in \nu$, or the equivalence class of a sequence of horizontal  sets $[(I_m \times \set {j_m})_{m \in \nu}]$ where $I_m\subseteq m$ and $j \in m+1$ for each $m \in \nu$.
The partial-algebra embedding $\theta$ maps $[(\set {i_m}\times J_m)_{m \in \nu}]$ to $\set{[(i_m)_{m \in \nu}]} \times \set{[(j_m)_{m \in \nu}]\mid \set{m\in\nu\mid j_m\in J_m}\in U}$, and it maps 
 $[(I_m \times \set {j_m})_{m \in \nu}]$ to $\set{[(i_m)_{m \in \nu}]\mid
 \set{m\in\nu\mid i_m\in I_m}\in U} \times \set{[(j_m)_{m\in\nu}]}$.

 It is easy to check that $\theta$ is a well-defined partial-algebra embedding. We limit ourselves to showing that if $a^\theta \plusdot b^\theta$ is defined in $\c A(\Pi_{m\in\nu}m/U,\; \Pi_{m\in\nu}(m+1)/U)$ then $a \plusdot b$ is defined in $\Pi_{m\in\nu}\c A(m, m+1)/U$, since it is this condition that distinguishes  partial-algebra embeddings from    embeddings of  relational structures.

 We prove the contrapositive. Suppose $a \plusdot b$ is undefined and let $[(a_m)_{m \in \nu}] = a$ and $[(b_m)_{m \in \nu}] = b$. Then we can find $S \in U$ such that one of the following two possibilities holds. One, for each $m \in S$ there exists $(i_m, j_m)$ belonging to both (a representative of) $a_m$ and (a representative of) $b_m$. Or two, for each $m \in S$ there exists $i_m \neq i'_m$ and $j_m \neq j'_m$ such that $(i_m, j_m)$ belongs to  (a representative of) $a_m$ and $(i'_m, j'_m)$ belongs to (a representative of) $b_m$.  Extend $(i_m)_{m \in S}, (j_m)_{m \in S}$ and, if appropriate, $(i'_m)_{m \in S}$ and $(j'_m)_{m \in S}$ to $\nu$-sequences arbitrarily. If the first alternative holds then $([(i_m)_{m \in \nu}], [(j_m)_{m \in \nu}])$ belongs to (representatives of) both $a^\theta$ and $b^\theta$. So $a^\theta \plusdot b^\theta$ is undefined since the representatives are non-disjoint. If the second alternative holds then $[(i_m)_{m \in \nu}] \neq [(i'_m)_{m \in \nu}]$, $[(j_m)_{m \in \nu}]\neq [(j'_m)_{m \in \nu}]$ and $([(i_m)_{m \in \nu}], [(j_m)_{m \in \nu}])$ belongs to (a representative of) $a^\theta$ and $([(i'_m)_{m \in \nu}], [(j'_m)_{m \in \nu}])$ belongs to (a representative of) $b^\theta$. So $a^\theta \plusdot b^\theta$ is undefined since the union of the representatives is not axial. 
 
We now argue that $\c A(\Pi_{m\in\nu}m/U,\; \Pi_{m\in\nu}(m+1)/U)$ is representable, by showing that the cardinalities of its two parameters are equal.   The map $f:\Pi_{m\in\nu}m/U \rightarrow \Pi_{m\in\nu}(m+1)/U$ defined by $f([(i_m)_{m\in\nu}])=[(i_m+1)_{m\in\nu}]$ is injective and its range is  all of $\Pi_{m\in\nu}(m+1)/U$ except $[(0, 0, \ldots)]$.  Since these are infinite sets it follows that the cardinality of $\Pi_{m\in\nu}m/U$ equals the cardinality of  $\Pi_{m\in\nu}(m+1)/U$.  It follows by \Cref{lem:mn} that $\c A(\Pi_{m\in\nu}m/U,\; \Pi_{m\in\nu}(m+1)/U)$ is $(\cupdot, \emptyset)$-representable.

Since the partial algebra $\Pi_{m\in\nu}\c A(m, m+1)/U$ has a partial algebra embedding into a representable partial algebra and the class of representable partial algebras is closed under partial subalgebras, we conclude that $\Pi_{m\in\nu}\c A(m, m+1)/U$ is itself representable. Hence we have an ultraproduct of unrepresentable partial algebras that is itself representable. It follows by \Los's theorem that the class of $(\cupdot, \emptyset)$-representable partial algebras cannot be defined by finitely many axioms.  
\end{proof}

\begin{corollary}\label{cor:S}  Let $\sigma$ be any one of the signatures $(\cupdot), (\dcupdot), (\cupdot, \emptyset), (\dcupdot, \emptyset), (\dcupdot, \mid)$ or $(\dcupdot, \mid, \emptyset)$.  The  class of $\sigma$-representable partial algebras is not finitely axiomatisable in ${\lang L(J)}$, ${\lang L(J, 0)}$, ${\lang L(J, \compo)}$ or $\lang L(J, \compo, 0)$, as appropriate. \end{corollary}

\begin{proof}
The case $\sigma = (\cupdot, \emptyset)$ is \Cref{thm:nfa}. The case $\sigma = (\dcupdot, \emptyset)$ follows by \Cref{prop}\eqref{first}, which tells us that the representation classes for $(\cupdot, \emptyset)$ and $(\dcupdot, \emptyset)$ coincide.

For the case $\sigma = (\dcupdot, \mid, \emptyset)$,  for any sets $m, n$ of cardinality greater than two, expand $\c A(m, n)$ to a partial $(\plusdot, \compo, 0)$-algebra $\c B(m, n)$ by defining $a \compo b = 0$ for all $a, b$. As in the proof of \Cref{thm:nfa}, write $\nu$ for $\omega \setminus \{0,1,2\}$ and let $U$ be a non-principal ultrafilter over $\nu$. Then for every  $m \in \nu$ the partial algebra $\c B(m, m+1)$ has no $(\dcupdot, \mid, \emptyset)$-representation, as its reduct to $(\plusdot, 0)$ has no $(\dcupdot, \emptyset)$-representation. However, as we saw in the proof of \Cref{thm:nfa}, the reduct of  $\Pi_{m\in\nu}\c B(m, m+1)/U$ to $(\plusdot, 0)$ does have a $(\dcupdot, \emptyset)$-representation and moreover, by \Los's theorem, it validates $a \compo b = 0$. By \Cref{prop}, these conditions ensure $\Pi_{m\in\nu}\c B(m, m+1)/U$ has a $(\dcupdot, \mid, \emptyset)$-representation. Once again we have an ultraproduct of unrepresentable partial algebras that is itself representable. Hence the representation class is not finitely axiomatisable.

For each of the signatures containing $\emptyset$ the result follows from the result for the corresponding signature without $\emptyset$, by \Cref{add0}. Because if the representation class for the signature without $\emptyset$ were finitely axiomatisable we could finitely axiomatise the case with $\emptyset$ by the addition of the single extra axiom $J(0,0,0)$.
\end{proof}

We can prove a stronger negative result about $\lesssim$-complete representability.

\begin{theorem}\label{thm:non elem}
The  class of $\lesssim$-completely $(\cupdot, \emptyset)$-representable partial algebras is not closed under elementary equivalence.
\end{theorem}

\begin{proof}
Consider the two partial $(\plusdot, \emptyset)$-algebras $\c A_1=\c A(\omega_1, \omega)$ and $\c A_0=\c A(\omega, \omega)$, where $\omega_1$ denotes the first uncountable ordinal.  By \Cref{lem:mn} the former is not $\lesssim$-completely $(\cupdot, \emptyset)$-representable while the latter is. 
 We prove these two partial algebras are elementarily equivalent by showing that the second player has a winning strategy in the Ehrenfeucht-Fr\"aiss\'e game of length $\omega$ played over $\c A_1$ and $\c A_0$.

Although elements of $\c A_1$ or $\c A_0$ are formally equivalence classes of axial sets, we may take $\set 0 \times \omega$ as the representative of $1$ and $\set i \times (\omega\setminus\set j)$ as the representative of $\overline{(i, j)}$, in either partial algebra.
 Since all elements are axial, each nonzero $a\in\c A_i$ uniquely determines (given this choice of representatives) sets $h_i(a)$ and $v_i(a)$ such that $a=h_i(a)\times v_i(a)$, for $i=0, 1$. For example $h_1(\set i\times J)=\set i,\;v_1(\set i\times J)=J,\; h_1(1)=\set0$ and $v_1(1)=\omega$. We will view $0$ as $\emptyset \times \emptyset$, in that $h_i (0) = v_i(0) = \emptyset$.

    For any sets $X, Y$ we write $X\approx Y$ if either
\begin{itemize}
\item
both $X$ and $Y$ contain $0$

\item
or neither contain $0$ and either $|X|=|Y|$ or both sets are infinite. 
\end{itemize}
   Observe, for any $X, Y$ and $U\subseteq X$, that
\begin{equation}\label{eq:sim}
X\approx Y\;\;\iff\;\;\mbox{ there is } V\subseteq Y\mbox{  with } U\approx V\mbox{ and }X\setminus U\approx Y\setminus V.
\end{equation}

 Initially there are no pebbles in play.  After $k$ rounds there will be $k$ pebbles on $\bar b=(b_0, \ldots, b_{k-1})\in\c A_1^k$ and $k$ matching pebbles on $\bar a=(a_0, \ldots, a_{k-1})\in\c A_0^k$.   For each $S\subseteq k$ let 
\begin{align*}
h_1(\bar b, S) &= \bigcap_{i\in S} h_1(b_i)\cap\bigcap_{i\in k\setminus S}(\omega_1\setminus h_1(b_i)),\\
v_1(\bar b, S) &= \bigcap_{i\in S} v_1(b_i)\cap\bigcap_{i\in k\setminus S}(\omega\setminus v_1(b_i)),
\end{align*}
with similar definitions for $h_0(\bar a, S)$ and $v_0(\bar a, S)$.    Observe that $\set{h_1(\bar b, S)\mid S\subseteq k}\setminus\set\emptyset$ is a finite partition of $\omega_1$ and each of $\set{v_1(\bar b, S)\mid S\subseteq k}\setminus\set\emptyset, \; \set{h_0(\bar a, S)\mid S\subseteq k}\setminus\set\emptyset$ and $\set{v_0(\bar a, S)\mid S\subseteq k}\setminus\set\emptyset$ is a finite partition of $\omega$.

  As an induction hypothesis we assume, for each $S\subseteq k$, that $h_1(\bar b, S)\approx h_0(\bar a, S)$ and $v_1(\bar b, S)\approx v_0(\bar a, S)$.
Initially, when $k=0$, the only subset of $k$ is $\emptyset$ and we have $h_1((\;), \emptyset)=\omega_1\approx\omega=h_0((\;), \emptyset)$ and $v_1((\;), \emptyset)=\omega=v_0((\;), \emptyset)$.

In round $k$, suppose $\forall$ picks $b_k\in\c A_1$.  The subsets of $k+1$ are $\set{ S \cup \set k\mid S\subseteq k} \cup \set{ S \mid S\subseteq k}$. For any $S\subseteq k$, since $h_0((a_0, \ldots, a_{k-1}), S)\approx h_1((b_0, \ldots, b_{k-1}), S)$, by \eqref{eq:sim} there is $X_S\subseteq h_0((a_0, \ldots, a_{k-1}), S)$ such that 
\begin{equation}\label{eq:xS}
\begin{aligned}
X_S&\approx h_1((b_0, \ldots, b_k), S\cup\set k),\\h_0((a_0, \ldots, a_{k-1}), S)\setminus X_S&\approx h_1((b_0, \ldots, b_k), S)\mbox{.}
\end{aligned}
\end{equation}
Similarly there is $Y_S\subseteq v_0((a_0, \ldots, a_{k-1}), S)$ such that $Y_S\approx v_1((b_0, \ldots, b_k), S\cup\set k)$ and $v_0((a_0, \ldots, a_{k-1}), S)\setminus Y_S\approx v_1((b_0, \ldots, b_k), S)$. Player $\exists$ lets $a_k$ be the element of $\c A_0$ represented by $(\bigcup_{S\subseteq k} X_S)\times (\bigcup_{S\subseteq k} Y_S)$, which is an axial set since $b_k$ is. In fact more is true: because $b_k$ is the representative of its equivalence class, $(\bigcup_{S\subseteq k} X_S)\times (\bigcup_{S\subseteq k} Y_S)$ will be the representative of \emph{its} equivalence class, so $h_0(a_k) = \bigcup_{S\subseteq k} X_S$ and $v_0(a_k) = \bigcup_{S\subseteq k} Y_S$. Then it follows that $h_0((a_0, \ldots, a_k), S\cup\set{k})=X_S$ and $h_0((a_0, \ldots, a_k), S)=h_0((a_0, \ldots, a_{k-1}), S)\setminus X_S$ and similar identities hold for the vertical components.  Hence, by \eqref{eq:xS}, the induction hypothesis is maintained.  Similarly if $\forall$ picks $a_k\in \c A_0$, we know $\exists$ can find $b_k\in \c A_1$ so as to maintain the induction hypothesis.

 We claim  the induction hypothesis ensures $\exists$ will not lose the play. 
 To prove that $\exists$ does not lose, we must prove that $\set{(a_i, b_i) \mid i < k}$ is a partial isomorphism from $\c A_1$ to $\c A_0$ for every $k$. That is, we must prove for any $i, j, l<k$ that
\begin{enumerate}
\item\label{en:0}
$b_i = 0 \iff a_i = 0$,
\item\label{en:ij}
$b_i = b_j \iff a_i = a_j$, 
\item\label{en:J}
$J(b_{i}, b_{j}, b_{l}) \iff J(a_{i}, a_{j}, a_{l})$.
\end{enumerate}
Conditions \eqref{en:0} and \eqref{en:ij} follow immediately from the induction hypothesis. 

 Given that \eqref{en:0} and \eqref{en:ij} hold, it follows that  \eqref{en:J} also holds whenever $0\in\set{b_i, b_j}$.  To prove \eqref{en:J} for the remaining cases, we assume $J(b_i, b_j, b_l)$ holds, where $0\not\in\set{b_i, b_j}$ and distinguish three  cases: $b_l=1, \; b_l=\overline{(i', j')}$ (for some $i' \in \omega_1,\; j' \in \omega$) and $b_l\not\in \set 1 \cup \set{\overline{(i', j')}\mid i' \in \omega_1, \; j' \in \omega}$. 

  For $b_l=1$ we have  $h_1(b_{l})=\set0,\; v_1(b_{l})=\omega$ and either $h_1(b_{i})=h_1(b_{j})$ is a singleton and $v_1(b_{i})\cupdot v_1(b_{j})=\omega$, or $v_1(b_{i})=v_1(b_{j})$ is a singleton and $h_1(b_{i})\cupdot h_1(b_{j})=\omega_1$.  The induction hypothesis shows that a similar condition holds for the vertical and horizontal components of $a_i, a_j, a_l$, hence $J(a_i, a_j, a_l)$ also holds. 

 For $b_l=\overline{(i', j')}$ we have $h_1(b_l)=\set{i'},\; v_1(b_l)=\omega\setminus\set{j'}$ and either $h_1(b_i)=h_1(b_j)=\set{i'}$ and $v_1(b_i)\cupdot v_1(b_j)=\omega\setminus\set{j'}$, or $v_1(b_i)=v_1(b_j)=\set{j'}$  and $h_1(b_i)\cupdot h_1(b_j)=\omega_1\setminus\set{i'}$.  Again, the induction hypothesis implies that a similar condition holds for the vertical and horizontal components of $a_i, a_j, a_l$, hence $J(a_i, a_j, a_l)$ holds. 

 When $b_l\not\in\set{1}\cup \set{ \overline{(i', j')}\mid i'<\omega_1,\; j'<\omega}$ (still with $0\not\in\set{b_i, b_j}$) then either $h_1(b_i)=h_1(b_j)=h_1(b_l)$ is a singleton and $v_1(b_{i})\cupdot v_1(b_{j})=v_1(b_{l})$, or a  similar case, with $h_1$ and $v_1$ swapped.  As before, an equivalent property holds on $a_i, a_j, a_l$ and $J(a_i, a_j, a_l)$ follows. This completes the argument that the implication $J(b_{i}, b_{j}, b_{l})\allowbreak\implies J(a_{i}, a_{j}, a_{l})$ is valid. The implication $J(a_{i}, a_{j}, a_{l})\implies J(b_{i}, b_{j}, b_{l})$ is similar.

As $\exists$ can win all $\omega$ rounds of the play,  the two structures $\c A_1$ and $\c A_0$ are elementarily equivalent. Hence the $\lesssim$-completely $(\cupdot, \emptyset)$-representable partial algebras are not closed under elementary equivalence.
\end{proof}

\begin{corollary}\label{elementary-corollary}  Let $\sigma$ be any one of the signatures $(\cupdot), (\dcupdot), (\cupdot, \emptyset), (\dcupdot, \emptyset), (\dcupdot, \mid)$ or $(\dcupdot, \mid, \emptyset)$.   The  class of $\lesssim$-completely $\sigma$-representable partial algebras is not closed under elementary equivalence.\end{corollary}

\begin{proof}
The case $\sigma = (\cupdot, \emptyset)$ is \Cref{thm:non elem}. For the case $\sigma = (\dcupdot, \emptyset)$, note that the proof used in \Cref{prop}\eqref{first} of the equivalence of representability by sets and by partial functions extends to $\lesssim$-complete representability. Hence the $\lesssim$-complete representation classes for $(\cupdot, \emptyset)$ and $(\dcupdot, \emptyset)$ coincide.

 For the case $\sigma = (\dcupdot, \mid, \emptyset)$, let $\c A_1, \c A_0$ be as defined in \Cref{thm:non elem}.  Expand $\c A_1$ and $\c A_0$  by adding a binary operation ${\compo}$ defined by $a \compo b=0$. It is clear that the two expansions are still elementarily equivalent since we have given the same first-order definition of $\compo$ for both. The expansion of $\c A_1$ does not have a $\lesssim$-complete $(\dcupdot, \mid, \emptyset)$-representation as $\c A_1$ itself is not completely representable. The expansion of $\c A_0$ does  have a $\lesssim$-complete $(\dcupdot, \mid, \emptyset)$-representation, which we can easily see via the same method employed in the proof of \Cref{prop}\eqref{mid}.

The results for signatures not including $\emptyset$ again follow straightforwardly from those for the corresponding signatures with $\emptyset$.  For a signature with $\emptyset$, take any elementarily equivalent $\c A_1, \c A_2$ with $\c A_1$ $\lesssim$-completely representable and $\c A_2$ not. Let $\c B_1, \c B_2$ be the reducts of $\c A_1, \c A_2$ to the signature without $0$. Then $\c B_1$ is $\lesssim$-completely representable since $\c A_1$ is. As $\c A_1$ is representable, it satisfies $J(0,0,0)$, so $\c A_2$ does too, by elementary equivalence. Now note that the content of \Cref{add0} applies to $\lesssim$-complete representability just as it does to representability. Hence if $\c B_2$ were $\lesssim$-completely representable then $\c A_2$ would have to be---a contradiction. Hence  $\c B_2$ is not $\lesssim$-completely representable. So for the signature without $\emptyset$ we have elementarily equivalent $\c B_1, \c B_2$ with the first $\lesssim$-completely representable and the second not.
\end{proof}

Finally we prove that all the negative results concerning representability for signatures containing $\cupdot$ carry over to signatures containing $\minusdot$.
First note that if a partial algebra $\c A=(A, \plusdot, \negdot)$ has a $(\cupdot,\minusdot)$-representation then it satisfies
\begin{equation}\label{eq:abc}a \negdot b = c \iff b \plusdot c = a.\end{equation} 
However as we see in the following example there exist partial $(\plusdot,\negdot)$-algebras satisfying \eqref{eq:abc}, whose $\plusdot$-reduct is $\cupdot$-representable but whose $\negdot$-reduct has no $\minusdot$-representation. Similarly there exist partial $(\plusdot,\negdot)$-algebras satisfying \eqref{eq:abc}, whose $\negdot$-reduct is $\minusdot$-representable but whose $\plusdot$-reduct has no $\cupdot$-representation.

\begin{example}\label{counterexample}
Our first partial algebra can be quite simple: a partial algebra consisting of a single element $a$, with $a \plusdot a$ and $a \negdot a$ both undefined. It satisfies \eqref{eq:abc} and is $\cupdot$-representable but not $\minusdot$-representable. Moreover, we give an example of a partial algebra containing a zero element. The domain is $\powerset\{1,2,3\} \setminus \{3\}$ and we define $\plusdot$ as $\cupdot$ and then define $\negdot$ using \eqref{eq:abc}. The identity map is a $\cupdot$-representation of the $\plusdot$-reduct (in fact, a $(\cupdot, \emptyset)$-representation of the $(\plusdot, 0)$-reduct).
Suppose $\theta$ is a $\minusdot$-representation of the $\negdot$-reduct. We show that $\{1\}^\theta \subseteq \{1,3\}^\theta$, which is a contradiction as $\{1, 3\} \negdot \{1\}$ is undefined. Let $x \in \{1\}^\theta$. Then $x \in \{1,2,3\}^\theta$ since $\{1,2,3\} \negdot \{1\}$ is defined.  As $\set{1,2}\negdot\set 1=\set 2$ and $x\in\set1^\theta$ we cannot have $x\in\set2$.   From $\set{1,2,3}\negdot\set2=\set{1,3}$ we deduce that  $x\in\set{1,3}^\theta$.

Similarly, if we take a partial algebra with domain $\powerset\set{1, 2,3}\setminus\set{1, 2,3}$, define $\negdot$ as $\minusdot$ and define $\plusdot$ using \eqref{eq:abc}, the identity map represents the $\negdot$-reduct, but the $\plusdot$-reduct of the partial algebra has no $\cupdot$-representation. To see this, note that since $\set{1,3}=\set1\plusdot\set 3$, in any $\cupdot$-representation $\set{1}$ and $\set3$ would have to  be represented by disjoint sets. By similar arguments, $\set1, \set2$ and $\set3$ would have to be represented by pairwise disjoint sets, contradicting the fact that $\set1\plusdot\set2\plusdot\set3$ is undefined.
\end{example}
Notwithstanding \Cref{counterexample} there is a simple condition that ensures a $\cupdot$-representation is always a $\minusdot$-representation, and vice versa.
\begin{definition}\label{def:complement}
A partial algebra $\c A=(A, \plusdot, \negdot, \ldots)$ is \defn{complemented} if it satisfies \eqref{eq:abc} and there is a unique $1 \in \c A$ such that $1\negdot a$ is  defined for all $a\in\c A$.  We write $\overline a$ for $1\negdot a$. 
\end{definition}
Observe  by \eqref{eq:abc} that
\begin{equation}
a\plusdot\overline a=1\label{eq:complement}\text{.}\end{equation}
Hence in a complemented partial algebra $\c A$, if $\sigma$ is a signature containing either $\cupdot$ or $\minusdot$ and $\theta$ is a $\sigma$-representation of $\c A$ then 
\begin{align}\label{eq:complement2}
{\overline a}^{\theta}=\overline{a^\theta}\text{,}
\end{align}
 where $\overline Y=1^\theta\minusdot Y$ for any $Y\subseteq 1^\theta$.

Before we articulate the consequences of a partial algebra being complemented, we describe a $\minusdot$-analogue of $\lesssim$-completeness. 
In any partial $(\ldots, \negdot, \ldots)$-algebra $\c A$, define a relation $\lesssim'$  by letting $a\lesssim' b$ if and only if $a =b$ or $b\negdot a$ is defined. If $\c A$ is $(\ldots, \minusdot, \ldots)$-representable then it is clear that $\lesssim'$ is a partial order. For $(\plusdot, \negdot)$ structures satisfying \eqref{eq:abc}, observe that $\lesssim'=\lesssim$.  

\begin{definition}
A subset $S$ of a partial $(\ldots, \negdot, \ldots)$-algebra $\c A$ is \defn{$\negdot$-pairwise-combinable} if for all distinct $s, t\in S$ there exists  $u\in\c A$ such that $u\negdot s=t$.
 As in \Cref{completely} we may define a $(\ldots, \minusdot, \ldots)$-representation to be \defn{$\lesssim'$-complete} if it maps $\lesssim'$-suprema of $\negdot$-pairwise-combinable sets  to (necessarily disjoint) unions.  
\end{definition}

\begin{lemma}\label{lem:comp}
Let $\c A=(A,\plusdot, \negdot, \ldots)$ be complemented and let $\theta$ be a map from $\c A$ to a subset of $\powerset (X)$ (for some $X$).  Then $\theta$ is a $\cupdot$-representation (of the $\plusdot$-reduct) if and only if it is a $\minusdot$-representation (of the $\negdot$-reduct).  Moreover, if $\theta$ is a representation it is $\lesssim$-complete if and only if it is  $\lesssim'$-complete.
\end{lemma}

\begin{proof}
Suppose $\c A$ is complemented and let $\theta$ be a $\minusdot$-representation.   For any $a, b\in \c A$, if $a \plusdot b$ is defined then by \eqref{eq:abc} we know that $(a \plusdot b) \negdot a = b$, which, by our hypothesis about $\theta$, implies that $a^\theta$ is disjoint from $b^\theta$, so $a^\theta \cupdot b^\theta$ is defined. We now show that, conversely, if $a^\theta \cupdot b^\theta$ is defined then $a \plusdot b$ is defined and 
$(a\plusdot b)^\theta = a^\theta\cupdot b^\theta$. Using equations to mean both sides are defined and  equal, assuming $a^\theta\cupdot b^\theta$ is defined, we have
\begin{align*}
 a^\theta\cap b^\theta&=\emptyset &&\mbox{by the definition of $\cupdot$}\\
 \overline a^\theta&\supseteq b^\theta&&\mbox{as }{\overline{a}}^\theta=\overline{a^\theta}\mbox{ and }b^\theta \subseteq 1^\theta\\
\overline{a} \negdot b&\mbox{ is defined} && \mbox{as }\theta\mbox{ is a $\minusdot$-representation}\\
\overline{\overline{a} \negdot b}&\mbox{ is defined} && \mbox{as }\c A\mbox{ is complemented}\\
{(\overline{\overline a\negdot b})}^\theta&=\overline{\overline{a^\theta} \minusdot b^\theta}&&\mbox{as }\theta\mbox{ is a $\minusdot$-representation and by \eqref{eq:complement2}}\\
&=a^\theta\cupdot b^\theta&&\mbox{by elementary set theory}\\
\overline{\overline a\negdot b}\negdot a&=b&&\mbox{as }\theta \mbox{ is a $\minusdot$-representation}\\
 a \plusdot b&=\overline{\overline a\negdot b}&&\mbox{by }\eqref{eq:abc}\\
 (a\plusdot b)^\theta &= a^\theta\cupdot b^\theta&&\mbox{by the calculation of }{(\overline{\overline a\negdot b})}^\theta\mbox{ above}
\end{align*}
and hence $\theta$ represents $\plusdot$ correctly as $\cupdot$. 

Dually, if $\theta$ is a $\cupdot$-representation and  $a \negdot b$ is defined then we know by \eqref{eq:abc} that $b \plusdot (a \negdot b)  = a$, implying  $a^\theta \minusdot b^\theta$ is defined. For the converse and for showing that when both are defined they are equal, assume $a^\theta\minusdot b^\theta$ is defined, so
\begin{align*} a^\theta&\supseteq b^\theta &&\mbox{by the definition of $\minusdot$}\\
 \overline a^\theta\cap b^\theta&=\emptyset&&\mbox{as }\overline a^\theta=\overline{a^\theta}\\
 \overline a\plusdot b&\mbox{ is defined}&&\mbox{as $\theta$ is a $\cupdot$-representation}\\
 (\overline a \plusdot b) \plusdot \overline{\overline a \plusdot b}& = 1 = \overline a \plusdot a&&\mbox{by \eqref{eq:complement}}\\
 b \plusdot  \overline{\overline a \plusdot b} &= a&&\mbox{cancelling the $\overline{a}$'s, as $\theta$ is a $\cupdot$-representation}\\
 (a\negdot b)^\theta&={\overline{\overline a\plusdot b}}^{\;\theta}&&\mbox{by \eqref{eq:abc}}\\
&={\overline {\overline{a^\theta}\plusdot b^\theta}}&&\mbox{as $\theta$ is a $\minusdot$-representation and by \eqref{eq:complement2}}\\
&=a^\theta\minusdot b^\theta&&\mbox{by elementary set theory}
\end{align*}
and so $\negdot$ is correctly represented as $\minusdot$.

For the final sentence of this lemma we do not need $\c A$ to be complemented, only that it validates \eqref{eq:abc}. Then the concepts `pairwise combinable' and `$\negdot$-pairwise combinable' coincide and the relations $\lesssim$ and $\lesssim'$ are equal. Hence the concepts `$\lesssim$-complete' and `$\lesssim'$-complete' coincide.
\end{proof}

\begin{theorem}\label{minus-theorem}  Suppose $\minusdot$ is included in $\sigma$ and all symbols in $\sigma$ are from $\set{\cupdot, \minusdot, \emptyset}$.
 The class of partial algebras $\sigma$-representable by sets is not finitely axiomatisable.  The class of partial algebras $\lesssim'$-completely $\sigma$-representable by sets is not closed under elementary equivalence.
\end{theorem}

\begin{proof}
 For $m, n$ of cardinality greater than two,  let $\c A'(m, n)$ be the expansion of $\c A(m, n)$ to $(\plusdot, \negdot, 0)$ where $\negdot$ is defined by \eqref{eq:abc}.    Observe that $\c A'(m, n)$ is complemented.   Let $\c A_\sigma(m, n)$ be the reduct of $\c A'(m, n)$ to the abstract analogue of $\sigma$.
By \Cref{lem:comp} (and the fact that $\c A'(m,n)$ satisfies $0 \plusdot 0 = 0$) we see that $\c A_\sigma(m, n)$ is $\lesssim'$-completely $\sigma$-representable if and only if $\c A'(m, n)$ is $\lesssim'$-completely $(\plusdot, \minusdot, \emptyset)$-representable, which is true if and only if $\c A(m, n)$ is $\lesssim$-completely $(\plusdot, \emptyset)$-representable. By \Cref{lem:mn} this is the case precisely when $|m|=|n|$. So $\c A_\sigma(m, m+1)$ is not $\sigma$-representable for $2<m<\omega$.

As before, write $\nu$ for $\omega \setminus \set{0,1,2}$ and let $U$ be any non-principle ultrafilter over $\nu$. We will argue that $\Pi_{m\in\nu}\c A_\sigma(m, m+1)/U$ is $\sigma$-representable. From $\Pi_{m\in\nu}\c A_\sigma(m, m+1)/U$, form the partial algebra $\c B'$ by expanding to $(\plusdot, \negdot, 0)$ using \eqref{eq:abc} and defining 0 in the obvious way, if necessary. Then let $\c B$ be the $(\plusdot, 0)$-reduct of $\c B'$. We can easily see that, $\c B'$ is complemented and in particular it validates \eqref{eq:abc}. Hence $\Pi_{m\in\nu}\c A_\sigma(m, m+1)/U$ is $\sigma$-representable if and only if $\c B$ is $(\plusdot, 0)$-representable. It is easy to check that $\c B = \Pi_{m\in\nu}\c A(m, m+1)/U$, which we know, by the proof of \Cref{thm:nfa}, is $(\plusdot, 0)$-representable. Hence the ultraproduct $\Pi_{m\in\nu}\c A_\sigma(m, m+1)/U$ of non-$\sigma$-representable partial algebras is itself $\sigma$-representable and so the class of $\sigma$-representable partial algebras is not finitely axiomatisable. 

For the second half of the theorem, we know, from the proof of \Cref{thm:non elem}, that $\c A(\omega_1, \omega)\equiv\c A(\omega, \omega)$, where $\equiv$ denotes elementary equivalence. Hence $\c A'(\omega_1, \omega)\equiv\c A'(\omega, \omega)$ since both expansions use the same first-order definition of $\negdot$. The elementary equivalence  of the reducts $\c A_\sigma(\omega_1, \omega)$ and $\c A_\sigma(\omega, \omega)$ follows. We established earlier in this proof that $\c A_\sigma(\omega_1, \omega)$ is not $\lesssim'$-completely $\sigma$-representable, while $A_\sigma(\omega, \omega)$ is. Hence the class of $\lesssim'$-completely $\sigma$-representable partial algebras is not closed under elementary equivalence.
\end{proof}

\begin{corollary}
Suppose $\minusdot$ is included in $\sigma$ and all symbols in $\sigma$ are from $\set{\dcupdot, \minusdot, \mid, \emptyset}$. The class of partial algebras $\sigma$-representable by partial functions is not finitely axiomatisable.  The class of partial algebras $\lesssim'$-completely $\sigma$-representable by partial functions is not closed under elementary equivalence.
\end{corollary}

\begin{proof}
\Cref{prop}\eqref{first} tells us that when all symbols are from $\set{\dcupdot, \minusdot, \emptyset}$, representability by partial functions is the same as representability by sets (with $\dcupdot$ in place of $\cupdot$). The proof of \Cref{prop}\eqref{first} extends to equality of $\lesssim'$-complete representability by partial functions and by sets. Hence for these signatures the results are immediate corollaries of \Cref{minus-theorem}.

For  signatures $\sigma$ including both $\mid$ and $\emptyset$ we use the same methods as in the proofs of \Cref{cor:S} and \Cref{elementary-corollary}. Let $\sigma^-$ be the signature formed by removing $\mid$ from $\sigma$. First we expand the partial algebras $\c A_{\sigma^-}(m, m+1)$ described in the proof of \Cref{minus-theorem} to a signature including $;$ by defining $a \compo b = 0$ for all $a, b$. The expanded partial algebras are not representable since the $\c A_{\sigma^-}(m, m+1)$'s are not. The ultraproduct of the expanded partial algebras validates $a \compo b = 0$, by \Los's theorem and so is representable, by the same method as in the proof of \Cref{prop}\eqref{mid}. This refutes finite axiomatisability. For $\lesssim'$-complete representability, again define $\compo$ by $a \compo b = 0$, to expand both of the two elementarily equivalent partial algebras $\c A_{\sigma^-}(\omega_1, \omega)$ and $\c A_{\sigma^-}(\omega, \omega)$. The expansions $\c B_1$ and $\c B_0$ remain elementarily equivalent and  the first is not $\lesssim'$-completely representable whilst the second is, by the same method as in the proof of \Cref{prop}\eqref{mid}.

The remaining cases are signatures including $\mid$ but not $\emptyset$, that is $(\minusdot, \mid)$ and $(\dcupdot, \minusdot, \mid)$. For these the results follow from the corresponding signatures that include $\emptyset$, by the now-familiar arguments involving \Cref{add0} and its generalisation to $\lesssim'$-complete representability.
\end{proof}

\section{Signatures Including Intersection}\label{meet}

In this section we consider signatures including a total operation $\bmeet$ to be represented as intersection.  In contrast to the results of the previous section, the classes of partial algebras representable by sets are finitely axiomatisable. This is true for all signatures containing intersection and with other operations members of $\set{\cupdot, \minusdot, \emptyset}$.  In order to control the size of this paper we do not consider representability by partial functions, only noting that the proofs in this section are not immediately adaptable to that setting.

We start with the signatures $(\cupdot, \cap, \emptyset)$ and $(\cupdot, \cap)$. Consider the following finite set {\bf Ax$(J, \bmeet,0)$} of $\lang L(J, \bmeet, 0)$ axioms.
\begin{description}
\item [$\plusdot$ is single valued] $J(a, b, c)\wedge J(a, b, c') \impl c=c'$
\item [$\plusdot$ is commutative] $J(a, b, c)\rightarrow J(b, a, c)$
\item[$\bmeet$-semilattice]  $\bmeet$ is commutative, associative and idempotent
\item[$\bmeet$ distributes over $\plusdot$] $J(b, c, d) \impl J(a\bmeet b, a\bmeet c, a\bmeet d)$
\item[$0$ is identity for $\plusdot$]  $J(a, 0, a)$ 
\item[domain of $\plusdot$]
$\exists c  J(a, b, c) \biimp a\bmeet b=0  $
\end{description}
Let {\bf Ax$(J, \bmeet)$}  be obtained from {\bf Ax$(J, \bmeet, 0)$} by replacing the axioms concerning $0$ 
 (the `$0$ is identity for $\plusdot$' and `domain of $\plusdot$' axioms) by the following axiom stating that either there exists an  element $z$ that acts like $0$, or else the partial operation $\plusdot$ is nowhere defined.
\begin{equation}\label{eq:z}
\begin{array}{l}
\exists z(\forall a J(a, z, a)\wedge \forall a,b  (a\bmeet b=z\biimp\exists c J(a, b, c))\\
\vee\\
 \forall a, b, c\; \neg J(a, b, c)
\end{array}
\end{equation}

\begin{theorem}\label{intersection-theorem}
The class of $(J, \bmeet, 0)$-structures that are $(\cupdot, \cap, \emptyset)$-representable  by sets  is axiomatised by \Ax{J, \bmeet, 0}.  The class of $(J, \bmeet)$-structures that are $(\cupdot, \cap)$-representable by sets is axiomatised by \Ax{J, \bmeet}.
\end{theorem}

\begin{proof}
We first give a quick justification for the axioms being sound in both cases. It suffices to argue that the axioms are sound for disjoint-union partial algebras of sets, with or without zero respectively.

Let $\c A$ be a disjoint-union partial algebra of sets with zero. We attend to each axiom of \Ax{J, \bmeet, 0} in turn.
\begin{description}
\item [$\plusdot$ is single valued] if $J(a, b, c)$ and $J(a, b, c')$ hold then $a \cupdot b$ is defined and is equal to both $c$ and $c'$. Hence $c = c'$.
\item [$\plusdot$ is commutative] $J(a,b,c)$ holds if and only if $a \cap b = \emptyset$ and $a \cup b = c$. By commutativity of intersection and union this is equivalent to the conjunction $b \cap a = \emptyset$ and $b \cup a = c$, which holds if and only if $J(b, a, c)$ holds.
\item[$\bmeet$-semilattice]  the easily verifiable facts that intersection is commutative, associative and idempotent are well known.
\item[$\bmeet$ distributes over $\plusdot$] if $J(b, c, d)$ then $b \cap c = \emptyset$, so certainly $(a \cap b) \cap (a \cap c) = \emptyset$. The other condition necessary for $J(a\bmeet b, a\bmeet c, a\bmeet d)$ to hold is that $(a \cap b) \cup (a \cap c) = a \cap d$. The left-hand side equals $a \cap (b \cup c)$ and by our hypothesis $b \cup c = d$, so we are done.
\item[$0$ is identity for $\plusdot$]  for any set $a$ we have $a \cap \emptyset = \emptyset$ and $a \cup \emptyset = a$, which are the two conditions needed to establish $J(a, 0, a)$.
\item[domain of $\plusdot$]
for any sets $a$ and $b$ there exists a set $c$ such that $J(a, b, c)$ if and only if $a \cupdot b$ is defined, which is true if and only if $a \cap b = \emptyset$.
\end{description}

Now let $\c A$ be a disjoint-union partial algebra of sets \emph{without} zero. It is clear that for all the axioms not concerning $0$ the above soundness arguments still hold. To see that axiom \eqref{eq:z} holds, note that if $\emptyset \in \c A$ then $\emptyset$ is an element $z$ that acts like $0$, as the first clause of \eqref{eq:z} asks for. Alternatively, if $\emptyset \notin \c A$, then for any sets $a$ and $b$ the intersection $a \cap b$, which is an element of $\c A$, must be nonempty. Hence $a \cupdot b$ is undefined and so for any $c$ we have $\neg J(a, b, c)$, meaning the second clause of \eqref{eq:z} holds.

\medskip

The sufficiency of the axioms is proved for $(J, \bmeet, 0)$-structures by a modification of the proof of Birkhoff's representation theorem for distributive lattices.
Assume that \Ax{J, \bmeet, 0} is valid on a $(J, \bmeet, 0)$-structure $\c A$. By the `$\plusdot$ is single valued' axiom we can view $\c A$ as a partial $(\plusdot, \bmeet, 0)$-algebra.  A \defn{filter} $F$ is a nonempty subset of $\c A$ such that $a\bmeet b\in F\iff( a\in F \mbox{ and } b\in F)$.   For any nonempty subset $S$ of $\c A$ let $\lr{S}$ be the filter generated by $S$, that is, $\set{a\in\c A\mid \exists s_1, s_2, \ldots, s_n\in S\; \mbox{(some finite $n$)},\; a\geq s_1\bmeet s_2\bmeet\ldots\bmeet s_n}$,  where $\leq$ is the partial ordering given by the $\bmeet$-semilattice.\footnote{There might be no `filter generated by the empty set', that is, no smallest filter, as the intersection of two or more filters can be empty.}  A filter is \defn{proper} if it is a proper subset of $\c A$. Recall that a set $F$ is $\plusdot$-prime if  $a\plusdot b\in F$ implies either $a\in F$ or $b\in F$.

Let $\Phi$ be the set of all  proper $\plusdot$-prime filters of $\c A$.  Define a  map $\theta$ from $\c A$ to $\powerset(\Phi)$ by letting $a^\theta=\set{F\in\Phi\mid a\in F}$. We will show that $\theta$ is a representation of $\c A$. 

The requirement that $(a\bmeet b)^\theta=a^\theta\cap b^\theta$ follows directly from the filter condition $a\bmeet b\in F\iff( a\in F \mbox{ and } b\in F)$. It follows easily from the axioms concerning $0$ that $0$ is the minimal element with respect to $\leq$. Hence a filter is proper if and only if it does not contain $0$. Then the requirement that $0^\theta = \emptyset$ follows directly from the condition that the filters in $\Phi$ be proper.

We next show that $\theta$ is faithful.  For this we show that if $a\not\leq b$ then there is a proper $\plusdot$-prime filter $F$ such that $a\in F$ but $b\not\in F$. The filters containing $a$ but not $b$, ordered by inclusion, form an inductive poset, that is, a poset in which every chain has an upper bound. (The empty chain has an upper bound since the up-set of $a$ is an example of a filter containing $a$ but not $b$.) Hence, by Zorn's lemma, there exists a maximal such filter, $F$ say.  We claim that $F$ is proper and $\plusdot$-prime.

Suppose, for contradiction, that $c \plusdot d$ is defined and belongs to $F$ but that neither $c \in F$ nor $d \in F$. By maximality of $F$ we have $b\in \lr{F\cup\set c}$ and $b\in \lr{F\cup\set d}$.  Then there is an $f\in F$ such that $f\bmeet c\leq b$ and $f\bmeet d\leq b$.  Then by the definition of $\leq$ and the distributive axiom, $b \bmeet f \bmeet (c \plusdot d) = (b \bmeet f \bmeet c) \plusdot (b \bmeet f \bmeet d) = (f \bmeet c) \plusdot (f \bmeet d) = f \bmeet (c \plusdot d)$. Hence $b \geq f \bmeet (c \plusdot d)$ and since both $f$ and $c \plusdot d$ are in $F$ we get that $b$ should be too---a contradiction.  Thus either $c \in F$ or $d \in F$. We conclude that $F$ satisfies the $\plusdot$-prime condition. Clearly $F$ is proper, as $b \not\in F$. Hence $F$ is a proper and $\plusdot$-prime filter and so $\theta$ is faithful.

To complete the proof that $\theta$ is a representation we show that $\plusdot$ is correctly represented as $\cupdot$. That is, $a\plusdot b$ is defined if and only if $a^\theta \cupdot b^\theta$ is defined, and when they are defined $(a\plusdot b)^\theta=a^\theta\cupdot b^\theta$. We have that
\begin{align*}
\mbox{$a \plusdot b$ is defined}& \iff  a \bmeet b = 0 && \mbox{by the domain of $\plusdot$ axiom}\\
&\iff  (a \bmeet b)^\theta = 0^\theta && \mbox{as $\theta$ is faithful}\\
&\iff  a^\theta \cap b^\theta = \emptyset && \mbox{as }0\mbox{ and}\bmeet\mbox{are represented correctly}\\
&\iff  a^\theta \cupdot b^\theta \mbox{ is defined} && \mbox{by the definition of $\cupdot$.}
\end{align*}
Further, when both $a\plusdot b$ and $a^\theta \cupdot b^\theta$ are defined it follows easily from \Ax{J, \bmeet, \emptyset} that $a=a\bmeet(a\plusdot b)$. So if $a$ is in a filter then  by the filter condition $a\plusdot b$ is too. Hence $a^\theta \subseteq (a \plusdot b)^\theta$, and similarly $b^\theta \subseteq (a \plusdot b)^\theta$, giving us $a^\theta \cupdot b^\theta \subseteq (a \plusdot b)^\theta$. By the $\plusdot$-prime condition on filters we get the reverse inclusion $(a\plusdot b)^\theta \subseteq a^\theta\cupdot b^\theta$. Hence $(a\plusdot b)^\theta=a^\theta\cupdot b^\theta$.

For a $(J, \bmeet)$-structure $\c A$, if \Ax{J, \bmeet} is valid in $\c A$ then \eqref{eq:z} holds.  If the first alternative of \eqref{eq:z} holds then we may form an expansion of $\c A$ to a $(J, \bmeet, 0)$-structure, interpreting $0$ as the $z$ given by this clause. Then by the above proof for $(J, \bmeet, 0)$-structures we can find a  $(\cupdot,\cap, \emptyset)$-representation of the expansion.  By ignoring the constant $0$ we obtain a $(\cupdot, \cap)$-representation of $\c A$. 

Otherwise, the second alternative in \eqref{eq:z} is true and $J(a, b, c)$ never holds, so we may define a representation $\theta$ of $\c A$ by letting $a^\theta=\set{b\in\c A\mid b\leq a}$.  Clearly $\theta$ represents $\bmeet$ as $\cap$ correctly, by the $\bmeet$-semilattice axioms.  Since $a\bmeet b\in a^\theta\cap b^\theta$ for any $a, b\in\c A$ and $a \plusdot b$ is never defined, $\theta$ also represents $\plusdot$ as $\cupdot$ correctly.
\end{proof}

From the previous theorem we can easily obtain finite axiomatisations for the signatures $(\cupdot, \minusdot, \cap, \emptyset)$ and $(\cupdot, \minusdot, \cap)$. Recall that we use the ternary relation $K$ to make first-order statements about the partial binary operation $\negdot$.

\begin{corollary}
The class of $(J, K, \bmeet, 0)$-structures that are $(\cupdot, \minusdot, \cap, \emptyset)$-representable by sets  is finitely axiomatisable.  The class of $(J, K, \bmeet)$-structures that are $(\cupdot, \minusdot, \cap)$-representable by sets is finitely axiomatisable.
\end{corollary}

\begin{proof}
To {\bf Ax$(J, \bmeet, 0)$} and {\bf Ax$(J, \bmeet)$} add the formulas $a \bmeet b = b \rightarrow \exists c K(a, b, c)$ and the relational form of \eqref{eq:abc} (that is, $K(a, b, c) \leftrightarrow J(b, c, a)$), which are valid on the representable partial algebras. Then when these axiomatisations hold, the representations in the proof of \Cref{intersection-theorem} will correctly represent $\negdot$ as $\minusdot$.
\end{proof}

We claimed finite representability for all signatures containing intersection and with other operations coming from $\set{\cupdot, \minusdot, \emptyset}$. For the signatures $(\cap)$ and $(\cap, \emptyset)$ finite axiomatisability is easy and well known. So the signatures remaining to be examined are $(\minusdot, \cap, \emptyset)$ and $(\minusdot, \cap)$.\footnote{As an aside, note these are  signatures for which representability by sets and by partial functions are easily seen to be the same thing.} Our treatment is very similar to the cases $(\cupdot, \cap, \emptyset)$ and $(\cupdot, \cap)$---no new ideas are needed---but we provide the details anyway. Consider the following finite set {\bf Ax$(K, \bmeet,0)$} of $\lang L(K, \bmeet, 0)$ axioms.
\begin{description}
\item [$\negdot$ is single valued] $K(a, b, c) \wedge K(a, b, c') \impl c=c'$
\item [$\negdot$ is left injective] $K(a, b, c) \wedge  K(a', b, c) \impl a = a'$
\item [$\negdot$ is subtractive] $K(a, b, c) \biimp K(a, c, b)$
\item[$\bmeet$-semilattice]  $\bmeet$ is commutative, associative and idempotent
\item[$\bmeet$ distributes over $\negdot$] $K(b, c, d) \impl K(a\bmeet b, a\bmeet c, a\bmeet d)$
\item[$0$ is identity for $\negdot$]  $K(a, 0, a)$ 
\item[domain of $\negdot$]
$\exists c  K(a, b, c) \biimp a\bmeet b=b$
\end{description}
Let {\bf Ax$(K, \bmeet)$}  be obtained from {\bf Ax$(K, \bmeet, 0)$} by replacing the `$0$ is identity for $\negdot$'  axiom by the  axiom \begin{equation}\label{z}\neg \exists a \vee \exists z \forall a K(a, z, a)\end{equation} stating that, provided $\c A$ is nonempty, there exists an element $z$ that acts like $0$.

\begin{theorem}\label{same}
The class of $(K, \bmeet, 0)$-structures that are $(\minusdot, \cap, \emptyset)$-representable  by sets is axiomatised by \Ax{K, \bmeet, 0}.  The class of $(K, \bmeet)$-structures that are $(\minusdot, \cap)$-representable by sets is axiomatised by \Ax{K, \bmeet}.
\end{theorem}

\begin{proof}
Again we give a quick justification for the soundness of the axioms. It suffices to argue that the axioms are sound for  partial $(\minusdot, \cap, \emptyset)$-algebras of sets and for partial $(\minusdot, \cap)$-algebras of sets   respectively.

Let $\c A$ be a partial $(\minusdot, \cap, \emptyset)$-algebra of sets. We attend to each axiom of \Ax{K, \bmeet, 0} in turn.
\begin{description}
\item [$\negdot$ is single valued] if $K(a, b, c)$ and $K(a, b, c')$ hold then $a \minusdot b$ is defined and is equal to both $c$ and $c'$. Hence $c = c'$.
\item [$\negdot$ is left injective] re-write the axiom with the predicate $J$, using \eqref{eq:abc}, then it becomes `$\plusdot$ is single valued', which we verified in \Cref{intersection-theorem}.
\item [$\negdot$ is subtractive] re-write  with $J$, then it becomes `$\plusdot$ is commutative'.
\item[$\bmeet$-semilattice] as in proof of \Cref{intersection-theorem}.
\item[$\bmeet$ distributes over $\negdot$] re-write  with $J$, then it becomes `$\bmeet$ distributes over $\plusdot$'.
\item[$0$ is identity for $\negdot$]  clear.
\item[domain of $\negdot$]
for any sets $a$ and $b$ there exists a set $c$ such that $K(a, b, c)$ if and only if $a \minusdot b$ is defined, which is true if and only if $b \subseteq a$, true if and only if $a \cap b = b$.
\end{description}

Now let $\c A$ be a partial $(\minusdot, \cap)$-algebra. It is clear that for all the axioms not concerning $0$ the above soundness arguments still hold. To see that axiom  \eqref{z} holds, note that either the clause $\neg \exists a$ holds or we can take any $a \in \c A$ and find that $a \minusdot a$ is defined and hence its value, $\emptyset$, is a member of $\c A$ and witnesses the existence of a $z$ such that $\forall a K(a, z, a)$.

\medskip

To prove the sufficiency of the axioms for $(K, \bmeet, 0)$-structures we use the same method employed in the proof of \Cref{intersection-theorem}. The definitions of the ordering $\leq$, of filters and of proper filters remain the same. This time however, we define a filter to be \defn{$\negdot$-prime} if  $a \in F$ and  $\exists a\negdot b$ together imply either $b \in F$ or $a \negdot b \in F$.

Similarly to before, $\Phi$ is the set of all proper $\negdot$-prime filters of $\c A$ and our representation will be the map $\theta$ from $\c A$ to $\powerset(\Phi)$ defined by $a^\theta=\set{F\in\Phi\mid a\in F}$. That $\bmeet$ is correctly represented as intersection is again immediate from the (unchanged) definition of a filter. It follows from the `$0$ is identity for $\negdot$' and `domain of $\negdot$' axioms that once again a filter is proper if and only if it does not contain $0$. Hence $0$ is represented correctly as the empty set.

To show that $\theta$ is faithful, given $a \not\leq b$, as before, we can find a maximal filter $F$ containing $a$ but not $b$ and we show $F$ is proper and $\negdot$-prime.  

Suppose, for contradiction, that $c \in F$ and $c \negdot d$ is defined but that neither $d \in F$ nor $c \negdot d \in F$. By maximality of $F$ we have $b\in \lr{F\cup\set d}$ and $b\in \lr{F\cup\set {c \negdot d}}$.  So there is an $f\in F$ such that $f\bmeet d\leq b$ and $f\bmeet (c \negdot d)\leq b$.  Then by the definition of $\leq$ and the distributive axiom, $b \bmeet f \bmeet c \negdot f \bmeet d = b \bmeet f \bmeet c \negdot b \bmeet f \bmeet d = b \bmeet f \bmeet (c \negdot d) =  f \bmeet (c \negdot d)= f \bmeet c \negdot f \bmeet d$. Then by left-injectivity of $\negdot$ we obtain $b \bmeet f \bmeet c = f \bmeet c$, that is, $b \geq f \bmeet c$. Since both $f$ and $c$ are in $F$ we see that $b$ should be too---a contradiction.  Thus either $d \in F$ or $c \negdot d \in F$. We conclude that $F$ satisfies the $\negdot$-prime condition. Clearly $F$ is proper, as $b \not\in F$. Hence $F$ is a proper $\negdot$-prime filter and so $\theta$ is faithful.

Finally, we show that $\negdot$ is correctly represented as $\minusdot$. We have that
\begin{align*}
\mbox{$a \negdot b$ is defined}& \iff  a \bmeet b = b && \mbox{by the domain of $\negdot$ axiom}\\
&\iff  (a \bmeet b)^\theta = b^\theta && \mbox{as $\theta$ is faithful}\\
&\iff  b^\theta \subseteq a^\theta && \mbox{as $\bmeet$ is represented correctly}\\
&\iff  a^\theta \minusdot b^\theta \mbox{ is defined} && \mbox{by the definition of $\minusdot$.}
\end{align*}
Further, when both $a \negdot b$ and $a^\theta \minusdot b^\theta$ are defined it follows easily from \Ax{K, \bmeet, \emptyset} that $a \negdot b =a\bmeet(a \negdot b)$. So if $a \negdot b$ is in a filter then  by the filter condition $a$ is too. Hence $(a \negdot b)^\theta \subseteq a^\theta$.  Similarly, it is easy to show that $(a \negdot b) \bmeet b = 0$, so if $a \negdot b$ is in a proper filter then $b$ is not. Hence $(a \negdot b)^\theta \cap  b^\theta =  \emptyset$, giving us $(a \negdot b)^\theta \subseteq a^\theta \minusdot b^\theta$. By the $\negdot$-prime condition on filters we get the reverse inclusion $(a\negdot b)^\theta \supseteq a^\theta\minusdot b^\theta$. Hence $(a\negdot b)^\theta=a^\theta\minusdot b^\theta$.

For a $(K, \bmeet)$-structure $\c A$, if \Ax{K, \bmeet} is valid in $\c A$ then \eqref{z} holds.  If the first alternative of \eqref{z} holds then $\c A$ is empty and so the empty function forms a representation. Otherwise, the second alternative in \eqref{z} is true. Then we may form an expansion of $\c A$ to a $(K, \bmeet, 0)$-structure, interpreting $0$ as the $z$ given by this clause. Then by the above proof for $(K, \bmeet, 0)$-structures we can find a  $(\minusdot,\cap, \emptyset)$-representation of the expansion.  By ignoring the constant $0$ we obtain a $(\minusdot, \cap)$ representation of $\c A$. 
\end{proof}

\section{Decidability and Complexity}\label{Decidability}

We finish with a discussion of the decidability and complexity of problems of representability and validity. We also highlight some still-open questions.

\begin{theorem}
The problem of determining whether a finite partial $\plusdot$-algebra has a disjoint-union representation is in {\bf NP}.
\end{theorem}

\begin{proof}
Given a finite partial $\plusdot$-algebra $\c A = (A, \plusdot)$, a non-deterministic polynomial-time algorithm based on the proof of \Cref{lem:game} runs as follows.  For each distinct pair $a\neq b$ it creates a set $S_{a, b}$  and for each pair $a, b$ where $a\plusdot b$ is undefined it creates a set $T_{a, b}$ (all these sets are initially empty).   Then for each $c\in\c A$, each set $S_{a, b}$ and each set $T_{a, b}$ it guesses whether $c\in S_{a, b}$ and whether $c\in T_{a, b}$.  Once this is done, the algorithm then verifies that 
exactly one of $a$ and $b$ belongs to $S_{a, b}$, that both $a$ and $b$ belong to $T_{a, b}$ and that each of these sets is a $\plusdot$-prime, bi-closed, pairwise incombinable set (to verify this for any single set takes quadratic time, in terms of the size of the input $(A, \plusdot)$).    This takes quartic time.   By \Cref{lem:game} this non-deterministic algorithm  solves the problem.
\end{proof}

\begin{problem}\label{complete}
Is the problem of determining whether a finite partial $\plusdot$-algebra has a disjoint-union representation {\bf NP}-complete?
\end{problem}

Now turning our attention to validity, let $s(\bar a), t(\bar a)$ be terms built from variables in $\bar a$ and the constant $0$, using $\plusdot$.  We take the view that the equation $s(\bar a)=t(\bar a)$ is valid if for every disjoint-union  partial algebra of sets with zero, $\c A$, and every assignment of the variables in $\bar a$ to sets in $\c A$, either both $s(\bar a)$ and $t(\bar a)$ are undefined or they are both defined and are equal. The following result is rather trivial but worth noting. It contrasts with \Cref{thm:nfa} by showing that the equational fragment of the first-order theory of partial $(\cupdot, \emptyset)$-algebras is an extremely simple object.

\begin{theorem}\label{validity}
The validity problem for $(\plusdot, 0)$-equations can be solved in linear time.
\end{theorem}

\begin{proof}
A $(\plusdot, 0)$-term is formed from variables and $0$, using $\plusdot$. Now $\cupdot$ is associative in the sense that either both sides of $(a \cupdot b) \cupdot c = a \cupdot (b \cupdot c)$ are defined and equal, or neither is defined. In the same sense, $\cupdot$ is also commutative. Hence, in representable partial $(\plusdot, 0)$-algebras, the bracketing and order of variables in a term  does not affect whether a term is defined, under a given variable assignment,  or the value it denotes when it is defined.  Similarly, any zeros occurring in a term may be deleted from the term without altering its denotation. If a variable $a$ occurs more than once in a term then the term can only be defined if $a$ is assigned the value $0$.  Hence an  equation $s(\bar a)=t(\bar a)$ is valid if and only if 
\begin{enumerate}[label=(\alph*)]
\item   the set of variables occurring in $s(\bar a)$ is the same as the set of variables occurring in $t(\bar a)$, 
\item
the set of variables occurring more than once in $s(\bar a)$ is the same as the set of variables occurring more than once in $t(\bar a)$.
\end{enumerate}
 This can be tested in linear time.
\end{proof}

\begin{problem}
Consider the set $\Sigma$ of all first-order $\lang L(J)$-formulas satisfiable over some disjoint-union partial algebras of sets.  Is this language decidable and if so, what is its complexity?  
\end{problem} 

We have seen that the class  of partial algebras with $\sigma$-representa\-tions by sets is not finitely axiomatisable, provided either $\cupdot$ or $\minusdot$ is in $\sigma$ and all symbols in $\sigma$ are from $\set{\cupdot, \minusdot, \emptyset}$, and the same negative result holds for representations by partial functions (with $\dcupdot$ in place of $\cupdot$). However when intersection is added to these signatures the representation classes \emph{are} finitely axiomatisable by sets. This leaves some cases in question, with regards to finite axiomatisability.

\begin{problem}\label{remaining}
  Determine whether the class of partial algebras $\sigma$-representable by partial functions is finitely axiomatisable for  signatures $\sigma$ containing $\cap$ and either $\dcupdot$ or $\minusdot$,  where symbols in $\sigma$ are from $\set{\dcupdot, \minusdot, \cap, |, \emptyset}$.  
\end{problem}

\bibliographystyle{abbrv}

\bibliography{robin}

\def\www{/\allowbreak}
\providecommand{\bysame}{\leavevmode\hbox to3em{\hrulefill}\thinspace}
\providecommand{\MR}{\relax\ifhmode\unskip\space\fi MR }
\providecommand{\MRhref}[2]{%
  \href{http://www.ams.org/mathscinet-getitem?mr=#1}{#2}
}
\providecommand{\href}[2]{#2}
\begin{thebibliography}{10}

\bibitem{berdine2005smallfoot}
Josh Berdine, Cristiano Calcagno, and Peter~W. O'Hearn, \emph{Smallfoot:
  Modular automatic assertion checking with separation logic}, International
  Symposium on Formal Methods for Components and Objects, Springer, 2005,
  pp.~115--137.

\bibitem{B48}
George Boole, \emph{The mathematical analysis of logic, being an essay towards
  a calculus of deductive reasoning}, MacMillan, Barclay, \& MacMillan,
  Cambridge and George Bell, London, 1947.

\bibitem{Brotherston-Kanovich:10}
James Brotherston and Max Kanovich, \emph{Undecidability of propositional
  separation logic and its neighbours}, 25th Annual IEEE Symposium on Logic in
  Computer Science, 2010, pp.~137--146.

\bibitem{calcagno2015moving}
Cristiano Calcagno, Dino Distefano, J{\'e}r{\'e}my Dubreil, Dominik Gabi,
  Pieter Hooimeijer, Martino Luca, Peter O'Hearn, Irene Papakonstantinou, Jim
  Purbrick, and Dulma Rodriguez, \emph{Moving fast with software verification},
  NASA Formal Methods Symposium, Springer, 2015, pp.~3--11.

\bibitem{doi:10.1080/14786445408647421}
Arthur Cayley, \emph{On the theory of groups, as depending on the symbolic
  equation $\theta^n =1$}, Philosophical Magazine \textbf{7} (1854), no.~42,
  40--47.

\bibitem{ChK90}
C.~C. Chang and H.~Jerome Keisler, \emph{Model theory}, 3rd ed., North-Holland,
  Amsterdam, 1990.

\bibitem{Gratzer:ua79}
George Gr{\"{a}}tzer, \emph{Universal algebra}, 2nd ed., Springer-Verlag, New
  York, 1979.

\bibitem{HH:book}
Robin Hirsch and Ian Hodkinson, \emph{Relation algebras by games}, Studies in
  Logic and the Foundations of Mathematics, {North-Holland}, 2002.

\bibitem{hirsch}
Robin Hirsch, Marcel Jackson, and Szabolcs Mikul{\'a}s, \emph{The algebra of
  functions with antidomain and range}, Journal of Pure and Applied Algebra
  \textbf{220} (2016), no.~6, 2214--2239.

\bibitem{invitation}
M{a}rcel Jackson and Tim Stokes, \emph{An invitation to {C}-semigroups},
  Semigroup Forum \textbf{62} (2001), no.~2, 279--310.

\bibitem{1182.20058}
Ma{r}cel Jackson and Tim Stokes, \emph{{Partial maps with domain and range:
  extending Schein's representation}}, Communications in Algebra \textbf{37}
  (2009), no.~8, 2845--2870.

\bibitem{DBLP:journals/ijac/JacksonS11}
Mar{c}el Jackson and Tim Stokes, \emph{Modal restriction semigroups: towards an
  algebra of functions}, International Journal of Algebra and Computation
  \textbf{21} (2011), no.~7, 1053--1095.

\bibitem{Jackson2003393}
{M}arcel Jackson and Tim Stokes, \emph{Agreeable semigroups}, Journal of
  Algebra \textbf{266} (2003), no.~2, 393--417.

\bibitem{completerep}
Brett McLean, \emph{Complete representation by partial functions for
  composition, intersection and antidomain}, Journal of Logic and Computation
  \textbf{27} (2017), no.~4, 1143--1156.

\bibitem{finiterep}
Brett McLean and Szabolcs Mikul\'as, \emph{The finite representation property
  for composition, intersection, domain and range}, International Journal of
  Algebra and Computation \textbf{26} (2016), no.~6, 1199--1216.

\bibitem{reynolds2002separation}
John~C. Reynolds, \emph{Separation logic: A logic for shared mutable data
  structures}, 17th Annual IEEE Symposium on Logic in Computer Science, 2002,
  pp.~55--74.

\bibitem{Schein1970}
Boris~M. Schein, \emph{Relation algebras and function semigroups}, Semigroup
  Forum \textbf{1} (1970), no.~1, 1--62.

\bibitem{Sto34}
Marshall~H. Stone, \emph{{B}oolean algebras and their applications to
  topology}, Proceedings of the National Academy of Sciences of the United
  States of America \textbf{20} (1934), no.~3, 197--202.

\end{thebibliography}

\end{document}